\documentclass[mathpazo]{arxiv}

\usepackage{accents}
\usepackage{amsmath, amssymb, amsthm}
\usepackage{graphicx}
\usepackage{booktabs}
\usepackage{multirow}
\usepackage{siunitx}

\newcommand{\bkR}{{\mathbb R}}
\newcommand\munderbar[1]{%
  \underaccent{\bar}{#1}}

\begin{document}
\title{Stability  of a leap-frog discontinuous Galerkin method for time-domain Maxwell's equations in anisotropic materials}

 \author{Ad{\'e}rito Ara{\'u}jo,
       S\'{\i}lvia  Barbeiro\corrauth and Maryam Khaksar Ghalati}
 \address{\ CMUC, Department of Mathematics,
University of Coimbra, Apartado 3008, EC Santa Cruz, 3001 - 501 Coimbra, Portugal. \\
           }
 \emails{{\tt alma@mat.uc.ptl} (A.~Ara{\'u}jo), {\tt silvia@mat.uc.pt} (S.~Barbeiro),
          {\tt maryam@mat.uc.pt} (M.~ Kh.~Ghalati)}

\begin{abstract}
In this work we discuss the numerical discretization of the  time-dependent Maxwell's equations using a fully explicit leap-frog type discontinuous Galerkin \linebreak method.  We present a sufficient condition for the stability, for cases of typical boundary conditions, either perfect electric, perfect magnetic or first order Silver-M\"{u}ller.  The bounds of the stability region point out the influence of not only  the mesh size but also the dependence on the choice of the numerical flux and the degree of the polynomials  used in the construction of the finite element space, making possible to balance accuracy and computational efficiency. In the model we consider heterogeneous anisotropic permittivity tensors which arise naturally in many applications of interest. Numerical results supporting the analysis are provided.
\end{abstract}


\maketitle

\section{Introduction}

Maxwell's equations are  a fundamental set of partial differential equations which describe electromagnetic wave interactions with materials.
The advantages of using discontinuous Galerkin time domain (DGTD) methods on the simulation of electromagnetic waves propagation, when compared with classical finite-difference time-domain methods, finite volume time domain methods or finite element time domain methods, have been reported by several authors (see { \it e.g.} \cite{fezoui2005} and references therein cited for an overview). DGTD methods gather many desirable features such as being able to achieve high-order accuracy and  easily handle complex geometries. Moreover, they are suitable for parallel implementation on modern multi-graphics processing units. Local refinement strategies can be incorporated due to the possibility of considering irregular meshes with hanging nodes and local spaces of different orders.

Despite the relevance of the anisotropic case in applications ({\it e.g.} \cite{bornwolf1999, Leonhardt110, Yeh2009}),  most of the formulation of  the DGTD methods present in the literature are restricted to isotropic materials (\cite{hesthaven2008, hesthaven2004, lu2004}).  Motivated by our application of interest described in  \cite{araujo2013, Santos2015}, in the present paper we consider a model with an heterogeneous  anisotropic permittivity tensor. The treatment of anisotropic materials within a DGTD framework was discussed for instance in \cite{fezoui2005} (with central fluxes) and in \cite{konig2010} (with upwind fluxes). The stability analysis of DGTD methods for Maxwell's equations was considered in \cite{fezoui2005}, where the scheme that is defined with the central fluxes leads to a locally implicit in time method in the case of Silver-M\"{u}ller absorbing boundary conditions, and \cite{li2012}, where the scheme is defined with the upwind fluxes leading to an implicit method. 
 Our derivation extends the results in  \cite{fezoui2005} and  \cite{li2012} to a fully explicit in time method for both cases,  central fluxes and upwind fluxes.
 
 In this paper we study a fully explicit scheme for the numerical solution of Maxwell's equations that uses a nodal DG method (\cite{hesthaven2004}) for the integration in space with an explicit leap-frog type method for the time integration. We present a rigorous proof of stability showing the influence of the mesh size, the choice of the numerical flux and choice of the degree of the polynomials used in the construction of the finite element space and the boundary conditions, which can be either perfect electric, perfect magnetic or first order Silver-M\"uler. The obtained stability condition is of practical relevance since it gives an easy way to balance stability and accuracy. In our approach we take into account the formulation in two space dimensions as well as its generalization for the 3D case.  
 
This paper consists in six sections after this introduction. In Section 2,  we state the problem and in Section 3 we describe the formulation of the numerical method for the two-dimensional problem. In Section 4 we derive the stability result for the method described in the previous section.  We illustrate the theoretical results with numerical examples in Section 5. In the last section we extend the stability results for the three dimensional case.


\section{The governing equations}

The electromagnetic field consists of coupled electric and magnetic fields, known as electric field intensity, $E$,  and magnetic induction, $B$. The effects of these two fundamental fields on matter, can be characterized by the electric displacement and the magnetic field intensity vectors, frequently denoted by $D$ and $H$, respectively. The knowledge of the material properties can be used to derive a useful relation between $D$ and $E$ and between $B$ and $H$. Here we will consider the constitutive relations of the form $D = \epsilon E$ and $B=\mu H$,
 where $\epsilon$ is the medium's electric permittivity and $\mu$ is the medium's magnetic permeability.

In three-dimensional spaces for heterogeneous anisotropic linear media with no source, these equations can be written in the form (\cite{fezoui2005})
\begin{equation}\label{maxwell3D}
\epsilon \frac{\partial E}{\partial t} = \mbox{curl } H,\qquad \mu \frac{\partial H}{\partial t} = -\mbox{curl } E.
\end{equation}
In a similar fashion to \cite{yee1966}, we  decompose the electromagnetic wave in a transverse electric (TE) mode and a transverse magnetic (TM) mode, this way reducing  the number of equations implemented in our model. This assumption is appropriate when studying {\it e.g.} truly 2D photonic crystals \cite{Dobson1999}  or the electrodynamic properties of 2D materials like graphene \cite{Mikhailov2007}.

In what follows we shall analyse the time domain Maxwell's equations in the transverse electric (TE) mode, as in \cite{konig2010}, where  the only non-vanishing components of the electromagnetic fields are $E_x$, $E_y$ and $H_z$. For this case, and assuming no conductivity effects, the equations in the non-dimensional form are
\begin{align}
  \epsilon\frac{\partial E}{\partial t}  &= \nabla \times H \quad \textrm{in } \Omega\times(0,T_f] \label{Maxwell TE1}\\
  \mu\frac{\partial H}{\partial t} &= -\mbox{curl } E\label{Maxwell TE3} \quad \textrm{in } \Omega\times(0,T_f],
\end{align}
where $E=(E_x,E_y)$ and $H=(H_z)$.  This equations are set and solved on the bounded polygonal domain $\Omega \subset \mathbb{R}^2.$ %
Note that we use the following notation for the   vector and scalar curl operators
$$ \nabla \times H = \left(\frac{\partial H_z}{\partial y},-\frac{\partial H_z}{\partial x}\right)^T, \quad \mbox{curl } E=\frac{\partial E_y}{\partial x}-\frac{\partial E_x}{\partial y}.$$
 The electric permittivity of the medium,   $\epsilon$ and  the magnetic permeability of the medium $\mu$ are varying in space, being
$\epsilon$ an anisotropic tensor
\begin{equation}\epsilon=\begin{pmatrix} \epsilon_{xx} & \epsilon_{xy} \\ \epsilon_{yx} & \epsilon_{yy} \end{pmatrix},\label{tensor}\end{equation}
while we consider  isotropic permeability $\mu$.
We assume that electric permittivity tensor $\epsilon$ is symmetric and uniformly positive definite for almost every $(x,y) \in \Omega$, and  it is uniformly bounded with a strictly positive lower bound, {\it i.e.}, there are constants  $\underline\epsilon>0$ and $\overline \epsilon> 0$ such that, for almost every $(x,y)\in \Omega$,
$$\munderbar\epsilon |\xi|^2\le\xi^T\epsilon(x,y)\xi\le \overline\epsilon |\xi|^2,\qquad \forall \xi\in\mathbb{R}^2.$$
We also assume that there are constants $\underline \mu>0$ and $\overline \mu> 0$  such that, for almost every $(x,y) \in \Omega$,
$$\underline\mu \le\mu(x,y) \le \overline\mu.$$

Let the  unit outward normal vector to the boundary be denoted by $n$.
We can define an effective permittivity (\cite{konig2010}) by
$$\epsilon_{eff}=\frac{\det(\epsilon)}{n^T \epsilon n},$$
that is used to characterize the speed with which a wave travels along the direction of the unit normal
$$c =\sqrt{\frac{n^T \epsilon n }{\mu \det(\epsilon)}}.$$

The model equations (\ref{Maxwell TE1})--(\ref{Maxwell TE3}) must be complemented by proper boundary conditions. Here we consider the most common, either  the perfect electric conductor boundary condition (PEC)
\begin{equation}\label{PEC}
  n\times E = 0\qquad \textrm {on } \partial\Omega,
\end{equation}
 the perfect magnetic conductor boundary condition (PMC), %
\begin{equation}\label{PMC}
 n\times H = 0\qquad \textrm {on } \partial\Omega,
\end{equation}
or the first order Silver-M\"{u}ller absorbing boundary condition
\begin{equation}\label{SM}
  n\times E = c\mu n \times (H\times n)\qquad \textrm {on } \partial\Omega.
\end{equation}

Initial conditions
$$E(x,y,0) = E_0(x,y)\quad\textrm{and}\quad H(x,y,0)= H_0(x,y) \quad \textrm{in }\Omega,$$
must also be provided.

We can write  Maxwell's equations (\ref{Maxwell TE1})--(\ref{Maxwell TE3}) in a conservation form
\begin{eqnarray} \label{q_eq}
Q \frac{\partial q}{\partial t}+ \nabla \cdot F(q)=0 \quad \textrm{in } \Omega\times(0,T_f],
\end{eqnarray}
with
  \begin{eqnarray*}
Q=\begin{pmatrix} \epsilon & 0 \\ 0 & \mu \end{pmatrix}, \quad
 q=\begin{pmatrix}E_x\\ E_y\\H_z\end{pmatrix} \quad \mbox{and } \quad
F(q)= \begin{pmatrix}0 &H_z \\ -H_z & 0\\ -E_y & E_x\end{pmatrix}^T,
\end{eqnarray*}
where $\nabla \cdot$ denotes the divergence operator.


\section{A leap-frog discontinuous Galerkin method}

The aim of this section is to derive our computational method. We will consider a nodal discontinuous Galerkin method for the space discretization and a leap-frog method for the time integration.

\subsection{The discontinuous Galerkin method}\label{DG_method}

Assume that the computational domain  $\Omega$  is partitioned into $K$  triangular elements  $T_k$ such that $\overline \Omega = \cup_k T_k$. For simplicity, we consider that the resulting mesh $\mathcal T_h$ is conforming, that is, the intersection of two elements is either empty or an edge.

Let $h_k$ be the diameter of the triangle $T_k \in \mathcal T_h$, and $h$ be the maximum element diameter,
$$h_k=\sup_{P_1,P_2 \in T_K} \|P_1-P_2\|, \quad h=\max_{T_k \in \mathcal T_h}\{h_k\}.$$
We assume that the mesh is regular in the sense that  there is a constant $\tau >0$ such that
\begin{equation}\label{tau}
\forall \, T_k \in \mathcal T_h, \quad \frac{h_k}{\tau_k} \leq \tau,
\end{equation}
where $\tau_k$ denotes the maximum diameter of a ball inscribed in $T_k$.

On each element $T_k$, the solution fields are approximated by  polynomials of degree less than or equal to $N$.
The global solution $q(x,y,t)$ is then assumed to be approximated by the piecewise $N$ order polynomials
$$q(x,y,t)\simeq \tilde q(x,y,t)= \bigoplus _{k=1}^K \tilde{q}_k(x,y,t),$$
defined as the direct sum of the $K$ local polynomial solutions $\tilde{q}_k(x,t)=(\tilde E_{x_k}, \tilde E_{y_k}, \tilde H_{z_k})$.
We use the notation
$\tilde E_x(x,y,t)=\bigoplus _{k=1}^K \tilde{E}_{x_k}(x,y,t)$, $\tilde E_y(x,y,t)$$=\bigoplus _{k=1}^K \tilde{E}_{y_k}(x,y,t)$,  $\tilde H_z(x,y,t)=\bigoplus _{k=1}^K \tilde{H}_{z_k}(x,y,t).$
The finite element space is then taken to be
$$V_{N}=\{v \in L^2(\Omega)^3: v|_{T_k} \in P_N(T_k)^3\},$$
where $P_N(T_k)$ denotes the space of polynomials of degree less than or equal to $N$ on $T_k$. %
The fields are expanded in terms of interpolating Lagrange polynomials $L_i(x,y)$,
$$ \tilde{q}_k(x,y,t)=\sum_{i=1}^{N_p} \tilde{q}_k(x_i,y_i,t)L_i(x,y)=\sum_{i=1}^{N_p} \tilde{q}_{ki}(t) L_i(x,y).$$
Here $N_p$ denotes the number of coefficients that are utilized, which is related with the polynomial order $N$ via $N_p=(N+1)(N+2)/2.$

In order to deduce the method, we start by multiplying equation (\ref{q_eq}) by test functions $v \in V_N$, usually the Lagrange polynomials, and integrate over each  element $T_k$. The next step is to employ one integration by parts and to substitute in the resulting contour integral the flux $F$ by a numerical flux $F^*$. Reversing the integration by parts yields
\begin{eqnarray*}
\int_{T_k} \left(Q \frac{\partial \tilde q}{\partial t} + \nabla \cdot F(\tilde q) \right) v(x,y) \, dx dy
= \int_{\partial T_k} n \cdot \left(F(\tilde q)-F^*(\tilde q)\right)  v(x,y) \, ds,\end{eqnarray*}
where $n$ is the outward pointing unit normal vector of the contour.

The approximate fields are allowed to be discontinuous across element boundaries.   In this way, we introduce the notation for the jumps of the field values across the interfaces of the elements, $[\tilde E]=\tilde E^--\tilde E^+$ and
$[\tilde H]=\tilde H^--\tilde H^+$, where the superscript $``+"$ denotes the neighboring element and the superscript $``-"$ refers to the local cell.  Furthermore we introduce, respectively, the cell-impedances and cell-conductances
$Z^{\pm}=\mu^{\pm}c^{\pm}$ and $Y^{\pm}=\left(Z^{\pm}\right)^{-1}$
where
$$c^{\pm} =\sqrt{\frac{n^T \epsilon^{\pm} n }{\mu^{\pm}\det(\epsilon^{\pm})}}.$$
At the outer cell boundaries we set $Z^+=Z^-$.

The coupling between elements is introduced via numerical flux, defined by
\begin{eqnarray*}
n \cdot \left(F(\tilde q)-F^*(\tilde q)\right)
&=&\begin{pmatrix}
 \frac{-n_y}{Z^++Z^-} \left(Z^+ [\tilde H_z]-\alpha \left(n_x[\tilde E_y]-n_y[\tilde E_x]\right)\right)\\
 \frac{n_x}{Z^++Z^-} \left(Z^+ [\tilde H_z]-\alpha \left(n_x[\tilde E_y]-n_y[\tilde E_x]\right)\right)\\
 \frac{1}{Y^++Y^-} \left(Y^+ \left(n_x[\tilde E_y]-n_y[\tilde E_x]\right)-\alpha[\tilde H_z]\right)
 \end{pmatrix}.
\end{eqnarray*}
The parameter $\alpha \in [0,1]$ in the numerical flux can be used to control dissipation. Taking $\alpha=0$ yields a non dissipative central flux while $\alpha=1$ corresponds to the classic upwind flux.

In order to discretize the boundary conditions we set $[\tilde E_x]=2\tilde E_x^-$, $[\tilde E_y]=2\tilde E_y^-$, $[\tilde H_z]=0$ and $[\tilde E_x]=0$, $[\tilde E_y]=0$,  $[\tilde H_z]=2\tilde H_z^-$, for PEC and PMC boundary conditions, respectively.  For Silver-M\"{u}ller absorbing boundary conditions, using the same kind of approach as in \cite{alvarez2014},  we consider, for upwind fluxes
$Z^-\tilde H_z^+ =n_x\tilde E_y^+-n_y\tilde E_x^+ $ or equivalently $\tilde H_z^+=Y^-(n_x\tilde E_y^+-n_y\tilde E_x^+)$
and, for central fluxes
$Z^-\tilde H_z^+ =( n_x\tilde E_y^--n_y\tilde E_x^-)$ and $Y^-(n_x\tilde E_y^+-n_y\tilde E_x^+)=\tilde H_z^-.$
This is equivalent to consider, for both upwind and central fluxes, $\alpha=1$ for numerical flux at the outer boundary and
$[\tilde E_x]=\tilde E_x^-$, $[\tilde E_y]=\tilde E_y^-$ and $[\tilde H_z]=\tilde H_z^-$.

\subsection{Time discretization}

To define a fully discrete scheme, we divide the time interval $[0,T]$ into $M$ subintervals by points $0=t^0<t^1<\cdots<t^M=T$, where $t^m=m \Delta t$, $\Delta t$ is the time step size and $T+ \Delta t/2 \leq T_f$.
The unknowns related to the electric field are approximated at integer time-stations $t^m$ and are denoted by  $\tilde E_k^m=\tilde E_k(.,t^m)$. The unknowns related to the magnetic field are approximated at half-integer time-stations $t^{m+1/2}=(m+\frac{1}{2}) \Delta t$ and are denoted by  $\tilde H_k^{m+1/2}=\tilde H_k(.,t^{m+1/2})$.
With the above setting, we can now formulate the leap-frog DG method: given an initial approximation $(\tilde E_{x_k}^{0}, \tilde E_{y_k}^{0}, \tilde H_{z_k}^{1/2}) \in V_N$, for each $m=0,1,\ldots, M-1$,   find $(\tilde E_{x_k}^{m+1}, \tilde E_{y_k}^{m+1}, \tilde H_{z_k}^{m+1/2}) \in V_N$ such that, $\forall (u_k, v_k, w_k) \in V_N$,
\begin{eqnarray}
&& \left(\epsilon_{xx} \frac{ \tilde E_{x_k}^{m+1}-\tilde E_{x_k}^{m}}{\Delta t} +\epsilon_{xy}  \frac{ \tilde E_{y_k}^{m+1}-\tilde E_{y_k}^{m}}{\Delta t} ,u_k \right)_{T_k} = \left(\partial_y \tilde H_{z_k}^{m+1/2},
   u_k\right) _{T_k} \nonumber \\
&&\qquad+   \left(\frac{-n_y}{Z^++Z^-}\left(Z^+ [\tilde H_{z}^{m+1/2}]-\alpha \left(n_x[\tilde E_{y}^{m}]-n_y[\tilde E_{x}^{m}]\right)\right) ,u_k\right)_{\partial T_k},
\label{Fdis1}\\
&&\left(\epsilon_{yx} \frac{ \tilde E_{x_k}^{m+1}-\tilde E_{x_k}^{m}}{\Delta t} + \epsilon_{yy} \frac{ \tilde E_{y_k}^{m+1}-\tilde E_{y_k}^{m}}{\Delta t}, v_k \right) _{T_k } = -\left( \partial_x \tilde H_{z_k}^{m+1/2} , v_k \right)_{T_k} \nonumber \\
&&\qquad+ \left(\frac{n_x}{Z^++Z^-} \left(Z^+ [\tilde H_{z}^{m+1/2}]-\alpha \left(n_x[\tilde E_{y}^{m}]-n_y[\tilde E_{x}^{m}]\right)\right) , v_k\right)_{\partial T_k},
\label{Fdis2}\\
&&\left ( \mu \frac{ \tilde H_{z_k}^{m+3/2}-\tilde H_{z_k}^{m+1/2}}{\Delta t}, w_k\right)_{T_k} =  \left( \partial_y \tilde E_{x_k}^{m+1}
- \partial_x \tilde E_{y_k}^{m+1} ,w_k \right)_{T_k} \nonumber \\
&&\qquad+ \left (\frac{1}{Y^++Y^-} \left(Y^+ (n_x[\tilde E_y^{m+1}]-n_y[\tilde E_x^{m+1}])-\alpha[\tilde H_z^{m+1/2}]\right), w_k\right)_{\partial T_k}, \nonumber \\ \label{Fdis3}
\end{eqnarray}
where
$(\cdot,\cdot)_{T_k}$ and $(\cdot,\cdot)_{\partial T_k}$ denote the classical  $L^2(T_k)$ and $L^2(\partial T_k)$  inner-products.
The boundary conditions are considered as described in the previous section.

We want to emphasize that the scheme (\ref{Fdis1})--(\ref{Fdis3}) is fully explicit in time, in opposition to \cite{li2012}, where the scheme is defined with the upwind fluxes involving the unknowns $E_{k}^{m+1}$ and $H_{k}^{m+3/2}$ and to \cite{fezoui2005}, where the scheme that is defined with the central fluxes leads to a locally implicit time method in the case of Silver-M\"{u}ller absorbing boundary conditions.


\section{Stability analysis}\label{section_stability_2D}

The aim of this section is to provide a sufficient condition for the $L^2$-stability of the leap-frog DG method (\ref{Fdis1})--(\ref{Fdis3}).

Choosing $u_k = \Delta t \tilde{E}_{x_k}^{[m+1/2]}$, $v_k = \Delta t \tilde{E}_{y_k}^{[m+1/2]}$ and $w_k = \Delta t \tilde{H}_{z_k}^{[m+1]}$, where
$\tilde{E}^{[m+1/2]}=$ \linebreak $\left(\tilde{E}^{m}+\tilde{E}^{m+1}\right)/2$ and $\tilde{H}^{[m+1]}=\left(\tilde{H}^{m+1/2}+\tilde{H}^{m+3/2}\right)/2$,  we have

\begin{eqnarray}\label{fv1}
&&\left(\epsilon \tilde{E}_k^{m+1},\tilde{E}_k^{m+1}\right)_{T_k}-\left(\epsilon \tilde{E}_k^{m},\tilde{E}_k^{m}\right)_{T_k}
=2 \Delta t \left(\nabla \times \tilde{H}_{z_k}^{m+1/2},\tilde{E}_k^{[m+1/2]}\right)_{T_k}\nonumber\\
&&\qquad+ 2\Delta t\left( \frac{-n_y}{Z^++Z^-} \left(Z^+ [\tilde H_{z}^{m+1/2}]-\alpha \left(n_x[\tilde E_{y}^{m}]-n_y[\tilde E_{x}^{m}]\right)  \right), \tilde E_{x_k}^{[m+1/2]}\right)_{\partial T_k}\nonumber\\
&&\qquad+2\Delta t\left( \frac{n_x}{Z^++Z^-} \left(Z^+ [\tilde H_{z}^{m+1/2}]-\alpha \left(n_x[\tilde E_{y}^{m}]-n_y[\tilde E_{x}^{m}]\right)  \right),
\tilde E_{y_k}^{[m+1/2]}\right)_{\partial T_k} \nonumber 
\end{eqnarray}
and
\begin{eqnarray} \label{fv2}
&&\left(\mu \tilde{H}_{z_k}^{m+3/2},\tilde{H}_{z_k}^{m+3/2}\right)_{T_k}-\left(\mu \tilde{H}_{z_k}^{m+1/2},\tilde{H}_{z_k}^{m+1/2}\right)_{T_k}
= -2\Delta t \left( \mbox{curl }  \tilde E_{k}^{m+1} ,\tilde{H}_{z_k}^{[m+1]}\right)_{T_k}\nonumber\\
&&\qquad+ 2\Delta t \left( \frac{1}{Y^++Y^-} \left(Y^+ \left(n_x[\tilde E_y^{m+1}]-n_y[\tilde E_x^{m+1}]\right)-\alpha[\tilde H_z^{m+1/2}]\right),
\tilde H_{z_k}^{[m+1]}\right)_{\partial T_k}.  \nonumber \end{eqnarray}
Using the identity,
$$
\left( \mbox{curl }  \tilde E_{k}^{m+1} ,\tilde{H}_{z_k}^{[m+1]}\right)_{T_k}= \left(\nabla \times \tilde{H}_{z_k}^{[m+1]},\tilde{E}_k^{m+1}\right)_{T_k}+ \left(    n_x \tilde{E}_{y_k}^{m+1}  - n_y \tilde{E}_{x_k}^{m+1} , \tilde{H}_{z_k}^{[m+1]} \right)_{\partial T_k},
$$
summing (\ref{fv1}) and  (\ref{fv2})  from $m=0$ to $m=M-1$, and integrating by parts, we get
\begin{eqnarray}
&&\left(\epsilon \tilde{E}_k^{M},\tilde{E}_k^{M}\right)_{T_k}+\left(\mu \tilde{H}_{z_k}^{M+1/2},\tilde{H}_{z_k}^{M+1/2}\right)_{T_k}=
\left(\epsilon \tilde{E}_k^{0},\tilde{E}_k^{0}\right)_{T_k}+\left(\mu \tilde{H}_{z_k}^{1/2},\tilde{H}_{z_k}^{1/2}\right)_{T_k}\nonumber\\
&&\qquad +\Delta t \left(\nabla \times \tilde{H}_{z_k}^{1/2},\tilde{E}_k^{0}\right)_{T_k}
- \Delta t \left(\nabla \times \tilde{H}_{z_k}^{M+1/2},\tilde{E}_k^{M}\right)_{T_k}
+2 \Delta t \sum_{m=0}^{M-1}   A_k^m,\label{cond}
 \end{eqnarray}
where
\begin{eqnarray*}
A_k^m &=& \left( \frac{-n_y}{Z^++Z^-} \left(Z^+ [\tilde H_{z}^{m+1/2}]-\alpha \left(n_x[\tilde E_{y}^{m}]-n_y[\tilde E_{x}^{m}]\right)  \right),
\tilde E_{x_k}^{[m+1/2]}\right)_{\partial T_k}\nonumber\\
&&+\left( \frac{n_x}{Z^++Z^-} \left(Z^+ [\tilde H_{z}^{m+1/2}]-\alpha \left(n_x[\tilde E_{y}^{m}]-n_y[\tilde E_{x}^{m}]\right)  \right),
\tilde E_{y_k}^{[m+1/2]}\right)_{\partial T_k}\\
&&+\left( \frac{1}{Y^++Y^-} \left(Y^+ \left(n_x[\tilde E_y^{m+1}]-n_y[\tilde E_x^{m+1}]\right)-\alpha[\tilde H_z^{m+1/2}]\right),
\tilde H_{z_k}^{[m+1]}\right)_{\partial T_k} \\
&&- \left(   n_x \tilde{E}_{y_k}^{m+1} - n_y\tilde{E}_{x_k}^{m+1}, \tilde{H}_{z_k}^{[m+1]}  \right)_{\partial T_k}.
 \end{eqnarray*}

Let us denote by  $F^{int}$ the set of internal edges and $F^{ext}$ the set of edges that belong to the boundary $\partial \Omega$. Let $\nu_k$ be the set of indices of the neighboring elements of $T_k$. For each $i \in \nu_k$, we consider the internal edge $f_{ik}=T_i \cap T_k$,  and we denote by $n_{ik}$ the unit normal oriented from $T_i$ towards $T_k$. For each boundary edge $f_{k}=T_k\cap \partial \Omega$,  $n_k$ is taken to be the unitary outer normal vector to $f_k$.
 Summing over all elements $T_k \in \mathcal T_h$ we obtain
  \begin{eqnarray*}
\sum_{T_k  \in \mathcal T_h} A_k^m = B_1^m+B_2^m,
 \end{eqnarray*}
 where $B_1^m=B_{11}^m+B_{12}^m+B_{13}^m$ with
   \begin{eqnarray}
B_{11}^m  & =& \sum_{f_{ik} \in F^{int} } \int_{f_{ik}} \Bigg(  \frac{-(n_y)_{ki}}{Z_i+Z_k} \left(Z_i [\tilde H_{z_k}^{m+1/2}]-\alpha \left((n_x)_{ki}[\tilde E_{y_k}^{m}]-(n_y)_{ki}[\tilde E_{x_k}^{m}]\right)  \right)
\tilde E_{x_k}^{[m+1/2]}\nonumber\\
&&\quad + \frac{-(n_y)_{ik}}{Z_i+Z_k} \left(Z_k [\tilde H_{zi}^{m+1/2}]-\alpha \left((n_x)_{ik}[\tilde E_{yi}^{m}]-(n_y)_{ik}[\tilde E_{xi}^{m}] \right)\right)
\tilde E_{xi}^{[m+1/2]}\nonumber\\
&&\quad  -\frac{Y_i (n_y)_{ki}}{Y_i+Y_k} [\tilde E_{x_k}^{m+1}]
\tilde H_{z_k}^{[m+1]}
- \frac{Y_k (n_y)_{ik}}{Y_i+Y_k} [\tilde E_{xi}^{m+1}]
\tilde H_{zi}^{[m+1]}\nonumber\\
&&\quad  + (n_y)_{ki} \tilde{E}_{x_k}^{m+1}  \tilde{H}_{z_k}^{[m+1]}
+ (n_y)_{ik}\tilde{E}_{xi}^{m+1}  \tilde{H}_{zi}^{[m+1]} \Bigg)\, ds,\label{B1m}
 \end{eqnarray}
\begin{eqnarray}
B_{12}^m&=&\sum_{f_{ik} \in F^{int} } \int_{f_{ik}} \Bigg(   \frac{(n_x)_{ki}}{Z_i+Z_k} \left(Z_i [\tilde H_{z_k}^{m+1/2}]-\alpha \left((n_x)_{ki}[\tilde E_{y_k}^{m}]-(n_y)_{ki}[\tilde E_{x_k}^{m}]\right)  \right)
\tilde E_{y_k}^{[m+1/2]}\nonumber\\
&&\quad +\frac{(n_x)_{ik}}{Z_i+Z_k} \left(Z_k [\tilde H_{zi}^{m+1/2}]-\alpha \left((n_x)_{ik}[\tilde E_{yi}^{m}]-(n_y)_{ik}[\tilde E_{xi}^{m}]\right)  \right)
\tilde E_{yi}^{[m+1/2]}\nonumber\\
&&\quad  + \frac{Y_i (n_x)_{ki}}{Y_i+Y_k} [\tilde E_{y_k}^{m+1}]
\tilde H_{z_k}^{[m+1]}
+\frac{Y_k (n_x)_{ik}}{Y_i+Y_k} [\tilde E_{yi}^{m+1}]
\tilde H_{zi}^{[m+1]}\nonumber\\
&&\quad  - (n_x)_{ki}\tilde{E}_{y_k}^{m+1}  \tilde{H}_{z_k}^{[m+1]}
 -  (n_x)_{ik}\tilde{E}_{yi}^{m+1}  \tilde{H}_{zi}^{[m+1]}  \Bigg)\, ds,\label{B2m}
 \end{eqnarray}
\begin{eqnarray}
 B_{13}^m=-
\sum_{f_{ik} \in F^{int} } \int_{f_{ik}} \Bigg(
\frac{\alpha}{Y_i+Y_k} [\tilde H_{z_k}^{m+1/2}]
\tilde H_{z_k}^{[m+1]}
+\frac{\alpha}{Y_i+Y_k} [\tilde H_{zi}^{m+1/2}]
\tilde H_{zi}^{[m+1]}
\Bigg)\, ds 
\label{B3m}
 \end{eqnarray}
and $B_2^m$ has  the terms related with the outer boundary
  \begin{eqnarray}
B_2^m
&=& \sum_{f_{k} \in F^{ext} } \int_{f_{k}}\Bigg(  \frac{-(n_y)_{k}}{2 Z_k} \left(Z_k [\tilde H_{z_k}^{m+1/2}]-\alpha \left((n_x)_{k}[\tilde E_{y_k}^{m}]-(n_y)_{k}[\tilde E_{x_k}^{m}]\right)  \right)
\tilde E_{x_k}^{[m+1/2]}\nonumber\\
&&\quad +\frac{(n_x)_{k}}{2 Z_k} \left(Z_k [\tilde H_{z_k}^{m+1/2}]-\alpha \left((n_x)_{k}[\tilde E_{y_k}^{m}]-(n_y)_{k}[\tilde E_{x_k}^{m}]\right)  \right)
\tilde E_{y_k}^{[m+1/2]}\nonumber\\
&&\quad +\frac{1}{2 Y_k} \left(Y_k \left((n_x)_{k}[\tilde E_{y_k}^{m+1}]-(n_y)_{k}[\tilde E_{x_k}^{m+1}]\right)-\alpha[\tilde H_{z_k}^{m+1/2}]\right)
\tilde H_{z_k}^{[m+1]}\nonumber\\
&&\quad -  \left( (n_x)_{k} \tilde{E}_{y_k}^{m+1} - (n_y)_{k}\tilde{E}_{x_k}^{m+1}   \right)  \tilde{H}_{z_k}^{[m+1]} \Bigg)\, ds.\label{B4m}
 \end{eqnarray}

\begin{lemma}\label{B13}
Let $B_{11}^m$, $B_{12}^m$ and $B_{13}^m$ be defined by (\ref{B1m}), (\ref{B2m}) and (\ref{B3m}), respectively, and $B_1^m=B_{11}^m+B_{12}^m+B_{13}^m$. Then
 \begin{eqnarray*}
\sum_{m=0}^{M-1}  B_1^m
&\le&
\sum_{f_{ik} \in F^{int} } \int_{f_{ik}} \frac{1}{4(Z_i+Z_k)}\Bigg(-\alpha\left((n_y)_{ki}[\tilde E_{x_k}^{0}]-(n_x)_{ki} [\tilde E_{y_k}^{0}]\right)^2\nonumber\\
&&\quad+2\left((n_x)_{ki}\left( Z_i
\tilde E_{y_k}^{0}+Z_k \tilde E_{yi}^{0}\right)- (n_y)_{ki}\left( Z_i
\tilde E_{x_k}^{0}+Z_k
\tilde E_{xi}^{0}\right)\right)[\tilde H_{z_k}^{1/2}]\nonumber\\
&&\quad +\alpha\left(\left((n_y)_{ki}[\tilde E_{x_k}^{M}]-(n_x)_{ki} [\tilde E_{y_k}^{M}]\right)^2 - \left([\tilde H_{z_k}^{1/2}]^2  -[\tilde H_{z_k}^{M+1/2}]^2 \right)\right)
\nonumber\\
&&\quad+ 2 \left((n_y)_{ki}\left( Z_i
\tilde E_{x_k}^{M}+Z_k
\tilde E_{xi}^{M}\right)-(n_x)_{ki}\left( Z_i
\tilde E_{y_k}^{M}+Z_k \tilde E_{yi}^{M}\right)\right)[\tilde H_{z_k}^{M+1/2}]\Bigg)
\, ds.\nonumber
 \end{eqnarray*}
\end{lemma}

\begin{proof} Since
\begin{equation}\label{ZY1} \frac{Z_i}{Z_i+Z_k}+\frac{Y_i}{Y_i+Y_k}=\frac{Z_k}{Z_i+Z_k}+\frac{Y_k}{Y_i+Y_k}=1\end{equation}
and
\begin{equation}\label{ZY2}\frac{Z_i}{Z_i+Z_k}=\frac{Y_k}{Y_i+Y_k},\qquad \frac{Z_k}{Z_i+Z_k}=\frac{Y_i}{Y_i+Y_k},\end{equation}  
summing from $m=0$ to $m=M-1$, we conclude that 
  \begin{eqnarray*}
\sum_{m=0}^{M-1}   B_{11}^m&=&
\sum_{f_{ik} \in F^{int} }\int_{f_{ik}} \frac{(n_y)_{ki}}{2(Z_i+Z_k)} \Bigg(- \left(Z_i \tilde E_{x_k}^{0} +Z_k \tilde E_{xi}^{0}\right)[\tilde H_{z_k}^{1/2}] \\
&&\quad + \alpha \left((n_x)_{ki}[\tilde E_{y_k}^{0}]-(n_y)_{ki}[\tilde E_{x_k}^{0}]  \right)
[\tilde E_{x_k}^{0}]\nonumber\\
&&\quad+ \alpha\sum_{m=0}^{M-1} \left((n_x)_{ki}[\tilde E_{y_k}^{m+1}]-(n_y)_{ki}[\tilde E_{x_k}^{m+1}]
+   (n_x)_{ki}[\tilde E_{y_k}^{m}]-(n_y)_{ki}[\tilde E_{x_k}^{m}]  \right)
[\tilde E_{x_k}^{m+1}]\nonumber\\
&&\quad+ \left(Z_i \tilde E_{x_k}^{M} + Z_k\tilde E_{xi}^{M}\right)[\tilde H_{z_k}^{M+1/2}] - \alpha \left((n_x)_{ki}[\tilde E_{y_k}^{M}]-(n_y)_{ki}[\tilde E_{x_k}^{M}]  \right)
[\tilde E_{x_k}^{M}]\Bigg)\, ds. \nonumber
 \end{eqnarray*}
 In the same way, for $B_{12}^m$  we have
  \begin{eqnarray*}
\sum_{m=0}^{M-1}   B_{12}^m&=&
\sum_{f_{ik} \in F^{int} } \int_{f_{ik}}\frac{(n_x)_{ki}}{2(Z_i+Z_k)} \Bigg(  \left(Z_i \tilde E_{y_k}^{0}+Z_k\tilde E_{yi}^{0}\right)[\tilde H_{z_k}^{1/2}]
-  \left((n_x)_{ki}[\tilde E_{y_k}^{0}]-(n_y)_{ki}[\tilde E_{x_k}^{0}]  \right)
[\tilde E_{y_k}^{0}]\nonumber\\
&&\quad- \alpha \sum_{m=0}^{M-1}\left( (n_x)_{ki}[\tilde E_{y_k}^{m+1}]-(n_y)_{ki}[\tilde E_{x_k}^{m+1}]  +(n_x)_{ki}[\tilde E_{y_k}^{m}]-(n_y)_{ki}[\tilde E_{x_k}^{m}]  \right)
[\tilde E_{y_k}^{m+1}]\nonumber\\
&&\quad- \left(Z_i \tilde E_{y_k}^{M}+Z_k\tilde E_{yi}^{M}\right)[\tilde H_{z_k}^{M+1/2}]+\alpha \left((n_x)_{ki}[\tilde E_{y_k}^{M}]-(n_y)_{ki}[\tilde E_{x_k}^{M}]  \right)
[\tilde E_{y_k}^{M}]\Bigg)\, ds, \nonumber
 \end{eqnarray*}
 and for $B_{13}^m$
 \begin{eqnarray*}
\sum_{m=0}^{M-1}   B_{13}^m=
-\sum_{m=0}^{M-1} \sum_{f_{ik} \in F^{int} }  \int_{f_{ik}} \frac{\alpha}{2(Y_i+Y_k)} [\tilde H_{z_k}^{m+1/2}]
\left([\tilde H_{z_k}^{m+1/2}]+[\tilde H_{z_k}^{m+3/2}]\right) \, ds.
\nonumber
 \end{eqnarray*}

Observing that, for general sequences $\{a^m\}$ and $\{b^m\}$, hold
 \begin{eqnarray*}
 \sum_{m=0}^{M-1}  \left(a^{m+1}  + a^{m} \right)a^{m+1}&=& \frac{1}{2}
\left(-(a^0)^2  +(a^M)^2 + \sum_{m=0}^{M-1}  \left(a^m+a^{m+1}\right)^2 \right), \end{eqnarray*}
 \begin{eqnarray*}
 \sum_{m=0}^{M-1}  \left(a^{m+1}  + a^{m} \right)b^{m+1}&=& \frac{1}{2}
\Bigg(-a^0b^0  +a^Mb^M+ \sum_{m=0}^{M-1}  \left(a^mb^m+2a^{m}b^{m+1}+a^{m+1}b^{m+1}\right) \Bigg),
\end{eqnarray*}
 we get
 \begin{eqnarray*}
&&\sum_{m=0}^{M-1}  \left( B_{11}^m+  B_{12}^m\right) \le \sum_{f_{ik} \in F^{int} } \int_{f_{ik}} \frac{1}{4(Z_i+Z_k)}\Bigg(-\alpha(n_y)_{ki}^2
 \left(-[\tilde E_{x_k}^0]^2  +[\tilde E_{x_k}^{M}]^2 \right)\\
&&\quad +\alpha(n_x)_{ki}(n_y)_{ki}\left(-[\tilde E_{x_k}^0][\tilde E_{y_k}^0]   +[\tilde E_{x_k}^{M}][\tilde E_{y_k}^{M}] \right)-2(n_y)_{ki} \left(Z_i \tilde E_{x_k}^{0} +Z_k \tilde E_{xi}^{0}\right)[\tilde H_{z_k}^{1/2}] \nonumber\\
&&\quad+2\alpha(n_y)_{ki} \left((n_x)_{ki}[\tilde E_{y_k}^{0}]-(n_y)_{ki}[\tilde E_{x_k}^{0}]  \right)
[\tilde E_{x_k}^{0}]+ 2(n_y)_{ki} \left(Z_i \tilde E_{x_k}^{M} + Z_k\tilde E_{xi}^{M}\right)[\tilde H_{z_k}^{M+1/2}]\nonumber\\
&&\quad- 2\alpha(n_y)_{ki}  \left((n_x)_{ki}[\tilde E_{y_k}^{M}]-(n_y)_{ki}[\tilde E_{x_k}^{M}]  \right)
[\tilde E_{x_k}^{M}]-\alpha(n_x)_{ki}^2\left(-[\tilde E_{y_k}^0]^2  +[\tilde E_{y_k}^{M}]^2 \right)\nonumber\\
&&\quad+ \alpha(n_x)_{ki}(n_y)_{ki}\left(-[\tilde E_{x_k}^0][\tilde E_{y_k}^0]   +[\tilde E_{x_k}^{M}][\tilde E_{y_k}^{M}] \right)+2(n_x)_{ki} \left(Z_i \tilde E_{y_k}^{0}+Z_k\tilde E_{yi}^{0}\right)[\tilde H_{z_k}^{1/2}]
\nonumber\\
&&\quad- 2\alpha(n_x)_{ki}  \left((n_x)_{ki}[\tilde E_{y_k}^{0}]-(n_y)_{ki}[\tilde E_{x_k}^{0}] \right)
[\tilde E_{y_k}^{0}]- 2(n_x)_{ki} \left(Z_i \tilde E_{y_k}^{M}+Z_k\tilde E_{yi}^{M}\right)[\tilde H_{z_k}^{M+1/2}]\nonumber\\
&&\quad+2\alpha(n_x)_{ki}  \left((n_x)_{ki}[\tilde E_{y_k}^{M}]-(n_y)_{ki}[\tilde E_{x_k}^{M}]  \right)
[\tilde E_{y_k}^{M}]\Bigg)\, ds. \nonumber
 \end{eqnarray*}
We also have
\begin{eqnarray*}
 \sum_{m=0}^{M-1}   B_{13}^m&=&-\sum_{f_{ik} \in F^{int} }
 \int_{f_{ik}} \frac{\alpha}{4(Y_i+Y_k)} \Bigg([\tilde H_{z_k}^{1/2}]^2  -[\tilde H_{z_k}^{M+1/2}]^2+\sum_{m=0}^{M-1}  \left([\tilde H_{z_k}^{m+1/2}]+[\tilde H_{z_k}^{m+3/2}] \right)^2\Bigg)
\, ds\nonumber\\
&\leq &-\sum_{f_{ik} \in F^{int} }
 \int_{f_{ik}} \frac{\alpha}{4(Y_i+Y_k)} \left([\tilde H_{z_k}^{1/2}]^2  -[\tilde H_{z_k}^{M+1/2}]^2 \right)
\, ds,\nonumber
 \end{eqnarray*}
which concludes the proof. 
\end{proof}

Let us now analyze the term $B_2^m$ for different kinds of boundary conditions.

 \begin{lemma}\label{B4}
 Let $B_2^m$ be defined by (\ref{B4m}). Then
  \begin{eqnarray*}
 \sum_{m=0}^{M-1}B_2^m
&\leq & \sum_{f_{k} \in F^{ext} } \int_{f_{k}} \frac{\beta_1}{4Z_k} \Bigg(  -\left((n_y)_{k}\tilde E_{x_k}^{0}-(n_x)_{k}\tilde E_{y_k}^{0} \right)^2+ \left((n_y)_{k}\tilde E_{x_k}^{M}-(n_x)_{k}\tilde E_{y_k}^{M} \right)^2\Bigg)\\
&&\quad+\frac{\beta_2}{2}\Bigg( \tilde H_{z_k}^{1/2}  \left((n_x)_{k}\tilde E_{y_k}^{0}-(n_y)_{k}\tilde E_{x_k}^{0} - \frac{\beta_3}{2Y_k} \tilde H_{z_k}^{1/2}  \right)\\
&&\quad-\tilde H_{z_k}^{M+1/2}\left((n_x)_{k}\tilde E_{y_k}^{M}-(n_y)_{k}\tilde E_{x_k}^{M} - \frac{\beta_3}{2Y_k} \tilde H_{z_k}^{M+1/2}\right)\Bigg)\, ds,
 \end{eqnarray*}
 where $\beta_1=\alpha, \beta_2=0$ for PEC, $\beta_1=0, \beta_2=1$, $\beta_3=\alpha$ for PMC,  and $\beta_1=\beta_2=\frac12$, $\beta_3=1$ for Silver-M\"uller boundary conditions.
\end{lemma}

\begin{proof} First we consider PEC boundary
conditions. We have
   \begin{eqnarray*}
B_2^m
&=& \sum_{f_{k} \in F^{ext} } \int_{f_{k}}\frac{ \alpha}{ Z_k}\Bigg(
 (n_y)_{k}  \left((n_x)_{k}\tilde E_{y_k}^{m}-(n_y)_{k}\tilde E_{x_k}^{m}\right)
\tilde E_{x_k}^{[m+1/2]}\\
&&- (n_x)_{k} \left((n_x)_{k}\tilde E_{y_k}^{m}-(n_y)_{k}\tilde E_{x_k}^{m}\right)
\tilde E_{y_k}^{[m+1/2]} \Bigg)\, ds.
 \end{eqnarray*}
  Summing from $m=0$ to $m=M-1$ we obtain
      \begin{eqnarray*}
  \sum_{m=0}^{M-1} B_2^m
&=& \sum_{f_{k} \in F^{ext} } \int_{f_{k}} \frac{\alpha}{4 Z_k}\Bigg( -\left((n_x)_{k}\tilde E_{y_k}^{0}-(n_y)_{k}\tilde E_{x_k}^{0} \right)^2  +\left((n_x)_{k}\tilde E_{y_k}^{M}-(n_y)_{k}\tilde E_{x_k}^{M} \right)^2\\
&&- 4\sum_{m=0}^{M-1}  \left((n_x)_{k}\tilde E_{y_k}^{[m+1/2]}
-(n_y)_{k}\tilde E_{x_k}^{[m+1/2]}\right)^2
\Bigg)\, ds\\
&\leq& \sum_{f_{k} \in F^{ext} } \int_{f_{k}}\frac{\alpha}{4 Z_k}\Bigg( -\left((n_x)_{k}\tilde E_{y_k}^{0}-(n_y)_{k}\tilde E_{x_k}^{0} \right)^2 + \left((n_x)_{k}\tilde E_{y_k}^{M}-(n_y)_{k}\tilde E_{x_k}^{M} \right)^2\Bigg) \, ds.
 \end{eqnarray*}
  For PMC boundary conditions we have
 \begin{eqnarray*}
B_2^m
&=& \sum_{f_{k} \in F^{ext} } \int_{f_{k}} \Bigg(
\tilde H_{z_k}^{m+1/2}\left((n_x)_{k} \tilde E_{y_k}^{[m+1/2]}-(n_y)_{k}\tilde E_{x_k}^{[m+1/2]}\right)\nonumber\\
&&- \left(\frac{\alpha }{ Y_k}\tilde H_{z_k}^{m+1/2}+ (n_x)_{k}\tilde E_{y_k}^{m+1}-(n_y)_{k}\tilde E_{x_k}^{m+1}\right)
\tilde H_{z_k}^{[m+1]}\Bigg)\, ds.
 \end{eqnarray*}
  Summing from $m=0$ to $m=M-1$ results
      \begin{eqnarray*}
 \sum_{m=0}^{M-1} B_2^m
&\le&\frac12 \sum_{f_{k} \in F^{ext} }\int_{f_{k}} \Bigg( \tilde H_{z_k}^{1/2}   \left((n_x)_{k}\tilde E_{y_k}^{0}-(n_y)_{k}\tilde E_{x_k}^{0} - \frac{\alpha}{2Y_k} \tilde H_{z_k}^{1/2}  \right)\\
&& -\tilde H_{z_k}^{M+1/2}\left((n_x)_{k}\tilde E_{y_k}^{M}-(n_y)_{k}\tilde E_{x_k}^{M} - \frac{\alpha}{2Y_k} \tilde H_{z_k}^{M+1/2}\right)\Bigg)\, ds.
 \end{eqnarray*}

 For Silver-M\"uller absorbing boundary conditions we have
    \begin{eqnarray*}
 B_2^m&=
& \frac12 \sum_{f_{k} \in F^{ext} }\int_{f_{k}}  \Bigg(
\left(-(n_y)_{k} \tilde H_{z_k}^{m+1/2}+\frac{(n_y)_{k}}{ Z_k}  \left((n_x)_{k}\tilde E_{y_k}^{m}-(n_y)_{k}\tilde E_{x_k}^{m}\right)\right)
\tilde E_{x_k}^{[m+1/2]}\nonumber\\
&&+ \left((n_x)_{k} \tilde H_{z_k}^{m+1/2}-\frac{(n_x)_{k}}{ Z_k}  \left((n_x)_{k}\tilde E_{y_k}^{m}-(n_y)_{k}\tilde E_{x_k}^{m}\right)\right)
\tilde E_{y_k}^{[m+1/2]}\nonumber\\
&&- \left(\frac{1}{Y_k}\tilde H_{z_k}^{m+1/2}+ (n_x)_{k}\tilde E_{y_k}^{m+1}-(n_y)_{k}\tilde E_{x_k}^{m+1}\right)
\tilde H_{z_k}^{[m+1]}\Bigg)\, ds.
 \end{eqnarray*}
 Summing from $m=0$ to $m=M-1$, and taking into account the previous cases, we deduce that
  \begin{eqnarray*}
  \sum_{m=0}^{M-1} B_2^m
&\leq& \sum_{f_{k} \in F^{ext} }\int_{f_{k}} \frac{1}{8 Z_k} \Bigg(  -\left((n_y)_{k}\tilde E_{x_k}^{0}-(n_x)_{k}\tilde E_{y_k}^{0} \right)^2 +\left((n_y)_{k}\tilde E_{x_k}^{M}-(n_x)_{k}\tilde E_{y_k}^{M} \right)^2\Bigg)\\
&&+\frac14 \Bigg( \tilde H_{z_k}^{1/2}   \left((n_x)_{k}\tilde E_{y_k}^{0}-(n_y)_{k}\tilde E_{x_k}^{0} - \frac{1}{2Y_k} \tilde H_{z_k}^{1/2}  \right)\\
&&-\tilde H_{z_k}^{M+1/2}\left((n_x)_{k}\tilde E_{y_k}^{M}-(n_y)_{k}\tilde E_{x_k}^{M} - \frac{1}{2Y_k} \tilde H_{z_k}^{M+1/2}\right)\Bigg)\, ds,
 \end{eqnarray*}
which concludes the proof. 
\end{proof}

\begin{theorem} \label{stab_theorem} Let us consider the leap-frog DG method (\ref{Fdis1})--(\ref{Fdis3}) complemented with the discrete boundary conditions defined in Section \ref{DG_method}. If the time step $\Delta t$ is such that
\begin{equation}\label{stab_cond}
\Delta t < \frac{\min\{\munderbar \epsilon, \munderbar\mu\}}{\max\{{C}_E, C_H\}} \min\{h_k\},
\end{equation}
where
$${C}_E=\frac12 C_{inv} N^2+C_{\tau}^2 (N+1)(N+2)\left(2+\beta_2+\frac{2\alpha+\beta_1}{2\min\{Z_k\}}\right),$$
$${C}_H= \frac12 C_{inv} N^2+C_{\tau}^2 (N+1)(N+2)\left(2+\beta_2+\frac{\alpha+\beta_2\beta_3}{\min\{ Y_k\}}\right), $$
 with $C_{\tau}$ defined by (\ref{inq_trace}) of Lemma \ref{trace}  and $C_{inv}$ defined by (\ref{inq_inv}) of Lemma \ref{inverse_ineq}, and  $\beta_1=\alpha, \beta_2=0$ for PEC, $\beta_1=0, \beta_2=1$, $\beta_3=\alpha$ for PMC,  and $\beta_1=\beta_2=\frac12$, $\beta_3=1$ for Silver-M\"uller boundary conditions, then the method is stable.
\end{theorem}
\begin{proof}
From (\ref{cond}) and the previous lemmata, considering  the Cauchy-Schwarz's inequality and taking into account that $Z_i/(Z_i+Z_k)<1$, we obtain
\begin{eqnarray*}
&&\sum_{T_k \in \mathcal T_h}\left( \left(\epsilon \tilde{E}_k^{M},\tilde{E}_k^{M}\right)_{T_k}+\left(\mu \tilde{H}_{z_k}^{M+1/2},\tilde{H}_{z_k}^{M+1/2}\right)_{T_k}\right)\\
&&\le\sum_{T_k \in \mathcal T_h}\left(\left(\epsilon \tilde{E}_k^{0},\tilde{E}_k^{0}\right)_{T_k}+\left(\mu \tilde{H}_{z_k}^{1/2},\tilde{H}_{z_k}^{1/2}\right)_{T_k}\right)\\
&&\qquad+\Delta t \sum_{T_k \in \mathcal T_h}\left(\|\nabla \times \tilde{H}_{z_k}^{1/2}\|_{L^2(T_k)}\|\tilde{E}_k^{0}\|_{L^2(T_k)} +\|\nabla \times \tilde{H}_{z_k}^{M+1/2}\|_{L^2(T_k)}\|\tilde{E}_k^{M}\|_{L^2(T_k)} \right)\\
&&\qquad+2\Delta t \sum_{f_{ik} \in F^{int} } \Big( \|
\tilde E^{M}_{k}\|_{L^2(f_{ik})}\|[\tilde H_{z_k}^{M+1/2}]\|_{L^2(f_{ik})}+\|\tilde E^{0}_{k}\|_{L^2(f_{ik})}\|[\tilde H_{z_k}^{1/2}]\|_{L^2(f_{ik})}\Big) \\
&&\qquad+\frac{\alpha \Delta t}{4\min\{Z_k\}}\sum_{f_{ik} \in F^{int} }  \| [\tilde E_{k}^{M}]\|^2_{L^2(f_{ik})} +\frac{\alpha\Delta t}{4 \min\{Y_k\}}\sum_{f_{ik} \in F^{int}}\| [\tilde H_{z_k}^{M+1/2}]\|_{L^2(f_{ik})}^2\\
&&\qquad+\frac{\beta_1\Delta t}{2 \min\{Z_k\}}\sum_{f_{k} \in F^{ext}}  \| \tilde E_k^M\|_{L^2(f_k)}^2 +\frac{\beta_2\beta_3\Delta t}{ \min\{Y_k\}}\sum_{f_{k} \in F^{ext}}\|\tilde H_{z_k}^{M+1/2}\|_{L^2(f_k)}^2   \\
&&\qquad+ 2\beta_2\Delta t\sum_{f_{k} \in F^{ext}}\left( \|\tilde H_{z_k}^{1/2}\|_{L^2(f_k)}\| \tilde E_{k}^{0}\|_{L^2(f_k)} + \|\tilde H_{z_k}^{M+1/2}\|_{L^2(f_k)} \| \tilde E_{k}^{M}\|_{L^2(f_k)}\right).
 \end{eqnarray*}
Using the inequality (\ref{inq_trace}) of Lemma \ref{trace} and the inequality  (\ref{inq_inv}) of Lemma \ref{inverse_ineq} (both in Appendix),  we get
\begin{eqnarray*}
&&\min\{\munderbar\epsilon,\munderbar\mu\}\left(\|\tilde{E}^{M}\|^2_{\Omega}+\|\tilde{H}_{z}^{M+1/2}\|^2_{\Omega}\right)\le\max\{\bar\epsilon, \bar\mu\}\left(\|\tilde{E}^{0}\|^2_{\Omega}+\|\tilde{H}_{z}^{1/2}\|^2_{\Omega}\right)\\
&&\qquad+\frac{\Delta t}{2} C_{inv} N^2 \max\left\{h_k^{-1} \right\} \left(\|\tilde{H}_{z}^{1/2}\|_{\Omega}^2 + \|\tilde{E}^{0}\|_{\Omega}^2 +\| \tilde{H}_{z}^{M+1/2}\|_{\Omega}^2+ \|\tilde{E}^{M}\|_{\Omega}^2 \right)\\
&&\qquad+C_{\tau}^2 (N+1)(N+2)\Delta t\max\left\{h_k^{-1} \right\}\left(2+\beta_2+\frac{2\alpha+\beta_1}{2\min\{Z_k\}}\right) \| \tilde E^M\|_{\Omega}^2\\
&&\qquad+C_{\tau}^2 (N+1)(N+2)\Delta t\max\left\{h_k^{-1} \right\}\left(2+\beta_2+\frac{\alpha+\beta_2\beta_3}{\min\{ Y_k\}}\right) \| \tilde H_{z}^{M+1/2}\|_{\Omega}^2\\
&&\qquad+C_{\tau}^2 (N+1)(N+2)\Delta t\max\left\{h_k^{-1} \right\}\left(2+\beta_2\right)\left(  \| \tilde E^0\|_{\Omega}^2+ \|\tilde H_{z}^{1/2}\|_{\Omega}^2 \right).
 \end{eqnarray*}
and so, taking ${C}_0=\frac12 C_{inv} N^2+ C_{\tau}^2 (N+1)(N+2)\left(2+ \beta_2\right),$
 \begin{eqnarray*}
&&\left(\min\{\munderbar \epsilon, \munderbar\mu\}-\Delta t \max\left\{h_k^{-1} \right\}\max\{{C}_E,  C_H\}\right)\left(\|\tilde{E}^{M}\|^2_{L^2(\Omega)}+\|\tilde{H}_{z}^{M+1/2}\|^2_{L^2(\Omega)}\right)\\
&&\qquad \le \left(\max\{\bar \epsilon, \bar\mu\}+\Delta t \max\left\{h_k^{-1} \right\}{C}_0\right)\left(\|\tilde{E}^{0}\|^2_{L^2(\Omega)}+\|\tilde{H}_{z}^{1/2}\|^2_{L^2(\Omega)}\right),
\end{eqnarray*}
which concludes the proof. 
\end{proof}

The stability condition (\ref{stab_cond}) shows that the method is conditionally stable, which is natural since we considered an explicit time discretization.
Furthermore, it discloses the influence of the values of $\alpha$, $h_{min}$ and $N$ on the bounds of the stable region. This is of utmost importance to balance accuracy {\it versus } stability.


  \section{Numerical results}

  In this section we present numerical results that support the  theoretical results derived in the previous section.

We can check numerically that  (\ref{stab_cond}) defines a sharp stability condition, in terms of the influence of $N$ and $h_{min}=\min\{h_k\}$.  In our experiments, we computed  $C$ that satisfies
\begin{equation}\label{C_stability}
\Delta t_{max} = \frac{C}{(N+1)(N+2)} h_{min},
\end{equation}
where
$\Delta t_{max}$ is the maximum observed value of $\Delta t$ such that the method is stable.
For these tests, the domain is the  square $\Omega = (-1,1)^2$,  the simulation final time is fixed at $T = 1$, we consider a  symmetric and positive definite anisotropic constant permittivity tensor (\ref{tensor}), with $\epsilon_{xx} = 5, \epsilon_{xy} = \epsilon_{yx} = 1 $ and $\epsilon_{yy} = 3 $, 
 and $\mu = 1$.
We consider equations (\ref{Maxwell TE1})--(\ref{Maxwell TE3}) with initial conditions $E_x(x,y,0) = 0,
E_y(x,y,0) = 0, H_z(x,y,\Delta t/2) = \cos(\pi x) \cos(\pi y)\cos(\omega \Delta t/2)$, where $  \omega = \pi \sqrt{\frac{1}{\epsilon_{xx}}+\frac{1}{\epsilon_{yy}}}$, in the case of PEC boundary conditions and \linebreak
$E_x(x,y,0) = 0,$
$E_y(x,y,0) = 0,$ $H_z(x,y,\Delta t/2) = \sin(\pi\Delta t/2)\sin(\pi x y) $ in the case of Silver-M\"{u}ller absorbing boundary conditions.

In Table \ref{CFL_PEC_CENT} and Table \ref{CFL_PEC_UPWIND} the results are computed for different mesh sizes, considering respectively central and upwind fluxes in the DG method,  for the case of PEC  boundary conditions, while in Table \ref{CFL_SM_CENT} and Table \ref{CFL_SM_UPWIND}, the results are computed for the case of Silver-M\"uller boundary conditions.

{
\begin{table}[h!]
\footnotesize
  \begin{tabular}{lcccccccccc}
    \toprule
    \multirow{2}{*}{$h_{min}$} &
      \multicolumn{2}{c}{$N=1$} &
      \multicolumn{2}{c}{$N=2$} &
      \multicolumn{2}{c}{$N=3$}&
      \multicolumn{2}{c}{$N=4$}&
      \multicolumn{2}{c}{$N=5$}\\
      & {$\Delta t_{max}$} & {$C$} & {$\Delta t_{max}$}  & {$C$} & {$\Delta t_{max}$}  & {$C$} &  {$\Delta t_{max}$} & {$C$}&  {$\Delta t_{max}$} & {$C$} \\
      \midrule
    0.5657 & 0.17 & 1.80  & 0.1 & 2.12   & 0.065 & 2.30  & 0.044  & 2.33  & 0.032 & 2.37    \\
    0.2828 & 0.088 & 1.87  & 0.05 & 2.12   & 0.031 & 2.20   & 0.021 & 2.23   & 0.016  & 2.37   \\
    0.1414& 0.044 & 1.87  & 0.024 & 2.04  & 0.015 & 2.12   & 0.01 & 2.12   & 0.0078  & 2.32   \\
    0.0707& 0.021 & 1.78  & 0.012 & 2.04   & 0.0078  & 2.20 & 0.0054  & 2.30  & 0.0038  & 2.26 \\
    0.0354& 0.01 & 1.70  & 0.006  & 2.04  & 0.0039  & 2.20   & 0.0027  & 2.30   & 0.0019 & 2.26   \\
    0.0177& 0.0054 & 1.83 & 0.003 & 2.04   & 0.0019  & 2.15   & 0.0013 & 2.21  & 0.00095 & 2.26   \\
    \bottomrule
  \end{tabular}
  \caption{$\Delta t_{max}$ such that the method is stable and  $C$ computed by (\ref{C_stability}) for PEC boundary conditions and central flux. }
\label{CFL_PEC_CENT}
\end{table}
}

{
\begin{table}[h!]
\footnotesize
  \begin{tabular}{lcccccccccc}
    \toprule
    \multirow{2}{*}{$h_{min}$} &
      \multicolumn{2}{c}{$N=1$} &
      \multicolumn{2}{c}{$N=2$} &
      \multicolumn{2}{c}{$N=3$} &
            \multicolumn{2}{c}{$N=4$} &
      \multicolumn{2}{c}{$N=5$} \\
      & {$\Delta t_{max}$} & {$C$} & {$\Delta t_{max}$}  & {$C$} & {$\Delta t_{max}$}  & {$C$} & {$\Delta t_{max}$}  & {$C$} & {$\Delta t_{max}$}  & {$C$}\\
      \midrule
    0.5657 & 0.10 & 1.06 & 0.056  & 1.19  & 0.034  & 1.20  & 0.023 & 1.22  & 0.016  & 1.19    \\
    0.2828 & 0.047 & 1.00 & 0.026  & 1.10  & 0.016 & 1.13  & 0.011 & 1.17 & 0.0081  &  1.20  \\
    0.1414& 0.023 & 0.98  & 0.012 & 1.02  & 0.008  & 1.13   & 0.0054  & 1.15  & 0.0039  & 1.16  \\
    0.0707& 0.011 & 0.93  & 0.0062  & 1.05   & 0.0039  & 1.10  & 0.0026  & 1.10   & 0.0019  & 1.13   \\
    0.0354& 0.0055  & 0.93  & 0.003 & 1.02   & 0.0019 & 1.07  & 0.0013 & 1.10   & 0.0009  & 1.07   \\
    0.0177& 0.0027  & 0.92  & 0.0015  & 1.02  & 0.0009  & 1.02 & 0.0006  & 1.02  & 0.0004 & 0.95    \\
    \bottomrule
  \end{tabular}
  \caption{$\Delta t_{max}$ such that the method is stable and  $C$ computed by (\ref{C_stability}) for PEC boundary conditions and upwind flux.}
\label{CFL_PEC_UPWIND}
\end{table}
}
%

%
{
\begin{table}[h!]
\footnotesize
  \begin{tabular}{lcccccccccc}
    \toprule
    \multirow{2}{*}{$h_{min}$} &
      \multicolumn{2}{c}{$N=1$} &
      \multicolumn{2}{c}{$N=2$} &
      \multicolumn{2}{c}{$N=3$}&
      \multicolumn{2}{c}{$N=4$}&
      \multicolumn{2}{c}{$N=5$}\\
      & {$\Delta t_{max}$} & {$C$} & {$\Delta t_{max}$}  & {$C$} & {$\Delta t_{max}$}  & {$C$} &  {$\Delta t_{max}$} & {$C$}&  {$\Delta t_{max}$} & {$C$} \\
      \midrule
    0.5657 & 0.18 & 1.91  & 0.1 & 2.12   & 0.064 & 2.26   & 0.044  & 2.33  & 0.031 & 2.30    \\
    0.2828 & 0.092 & 1.95  & 0.05 & 2.12   & 0.031 & 2.19    & 0.021 & 2.02    & 0.015 & 2.23   \\
    0.1414& 0.044 & 1.87  & 0.024 & 2.04  & 0.015 & 2.12   & 0.01 & 2.12   & 0.0079  & 2.35   \\
    0.0707& 0.021 & 1.78  & 0.012 & 2.04   & 0.0077  & 2.18  & 0.0053  & 2.25  & 0.0038  & 2.26  \\
    0.0354& 0.01 & 1.70  & 0.006  & 2.04  & 0.0038  & 2.15    & 0.0026  & 2.21   & 0.0019  & 2.26    \\
    0.0177& 0.0053 & 1.80 & 0.003 & 2.04   & 0.0018  & 2.04   & 0.0012 & 2.04   & 0.00095 & 2.26   \\
    \bottomrule
  \end{tabular}
  \caption{$\Delta t_{max}$ such that the method is stable and  $C$ computed by (\ref{C_stability}) for SM boundary conditions and central flux. }
\label{CFL_SM_CENT}
\end{table}
}
{
\begin{table}[h!]
\footnotesize
  \begin{tabular}{lcccccccccc}
    \toprule
    \multirow{2}{*}{$h_{min}$} &
      \multicolumn{2}{c}{$N=1$} &
      \multicolumn{2}{c}{$N=2$} &
      \multicolumn{2}{c}{$N=3$} &
            \multicolumn{2}{c}{$N=4$} &
      \multicolumn{2}{c}{$N=5$} \\
      & {$\Delta t_{max}$} & {$C$} & {$\Delta t_{max}$}  & {$C$} & {$\Delta t_{max}$}  & {$C$} & {$\Delta t_{max}$}  & {$C$} & {$\Delta t_{max}$}  & {$C$}\\
      \midrule
    0.5657 & 0.11 & 1.17 & 0.057    & 1.21   & 0.035   & 1.24  & 0.023 & 1.22  & 0.016  & 1.19    \\
    0.2828 & 0.051 & 1.08  & 0.026  & 1.10  & 0.016 & 1.13  & 0.011 & 1.17 & 0.008  &  1.19  \\
    0.1414& 0.023 & 0.98  & 0.012 & 1.02  & 0.008  & 1.13   & 0.0054  & 1.15  & 0.0039  & 1.16  \\
    0.0707& 0.011 & 0.93  & 0.0061  & 1.04   & 0.0039  & 1.10  & 0.0026  & 1.10   & 0.0019  & 1.13   \\
    0.0354& 0.0055  & 0.93  & 0.003 & 1.02   & 0.0018 & 1.07  & 0.0013 & 1.10   & 0.00097  & 1.15    \\
    0.0177& 0.0027  & 0.92  & 0.0015  & 1.02  & 0.00097  & 1.10  & 0.00065  & 1.10   & 0.00045 & 1.07    \\
    \bottomrule
  \end{tabular}
  \caption{$\Delta t_{max}$ such that the method is stable and  $C$ computed by (\ref{C_stability}) for SM boundary conditions and upwind flux.}
\label{CFL_SM_UPWIND}
\end{table}
}

As expected from the condition (\ref{stab_cond}), the numerical examples show that the stability regions corresponding to
central fluxes are slightly bigger when compared to the regions obtained using upwind fluxes.
From all the examples presented, we may deduce  that the right hand side of (\ref{stab_cond}) is a sharp bound for $\Delta t_{max}$.
Moreover, we can also conclude that $\Delta t_{max}$ is directly  proportional $h_{min}$ and inversely proportional to $(N+1)(N+2)$.


\section{Stability of the 3 D model}

In this section we extend the analysis in Section \ref{section_stability_2D} of the TE form of Maxwell's equations in two-dimensions to the full three-dimensional time-dependent Maxwell equations (\ref{maxwell3D}), with the equations are set  on a bounded polyhedral domain $\Omega \subset \bkR^3$.
We can write the model in a conservation form (\ref{q_eq}) 
with
  \begin{eqnarray*}
Q=\begin{pmatrix} \epsilon & 0 \\ 0 & \mu \end{pmatrix}, \quad
 q=\begin{pmatrix}E\\ H\end{pmatrix} \quad \mbox{and } \quad
F(q)= \begin{pmatrix} n \times H\\ - n\times E\end{pmatrix},
\end{eqnarray*}
where $E=(E_x,E_y,E_z),$ $H=(H_x,H_y,H_z)$ and these are functions of $(x,y,z,t)$.

We assume that electric permittivity and the magnetic permeability tensors $\epsilon$ and $\mu$ are symmetric and uniformly positive definite for almost every $(x,y,z) \in \Omega$, and  are uniformly bounded with a strictly positive lower bound, {\it i.e.}, there are constants  $\underline\epsilon>0$,   $\overline \epsilon> 0$ and $\underline\mu>0$, $\overline \mu> 0$ such that, for almost every $(x,y,z)\in \Omega$,
$$\munderbar\epsilon |\xi|^2\le\xi^T\epsilon(x,y,z)\xi\le \overline\epsilon |\xi|^2,\qquad \munderbar\mu |\xi|^2\le\xi^T\mu(x,y,z)\xi\le \overline\mu |\xi|^2,\qquad \forall \xi\in\mathbb{R}^3.$$
Let us define an effective permeability (in the same way as the effective permittivity) by
$$\mu_{eff}=\frac{\det(\mu)}{n^T \mu n}.$$
Now the the speed with which a wave travels along the direction of the unit normal is given by
$$c =\sqrt{\frac{1}{\mu_{eff} \epsilon_{eff}}}.$$

We assume that $\Omega$ is partitioned into $K$ disjoint tetrahedral elements $T_k$. The leap-frog discontinuous Galerkin method is the natural extension of the formulation (\ref{Fdis1})-(\ref{Fdis3}) to the three-dimensional domain,  
with the numerical flux defined by
\begin{eqnarray*}
n \cdot \left(F(\tilde q)-F^*(\tilde q)\right)
&=&\begin{pmatrix}
 \frac{-1}{Z^++Z^-} n \times \left(Z^+ [\tilde H]-\alpha n \times [\tilde E] \right)\\
 \frac{1}{Y^++Y^-} n \times \left(Y^+ [\tilde E]+\alpha n \times [\tilde H]\right)
 \end{pmatrix}.
\end{eqnarray*}

We start by noticing that the following inequalities hold
\begin{eqnarray}\label{fv13D}
&&\left(\epsilon \tilde{E}_k^{m+1},\tilde{E}_k^{m+1}\right)_{T_k}-\left(\epsilon \tilde{E}_k^{m},\tilde{E}_k^{m}\right)_{T_k}
=2 \Delta t \left(\nabla \times \tilde{H}_{k}^{m+1/2},\tilde{E}_k^{[m+1/2]}\right)_{T_k}\nonumber\\
&&\qquad- 2\Delta t\left( \frac{1}{Z^++Z^-} n \times \left(Z^+ [\tilde H^{m+1/2}]-\alpha n \times [\tilde E^{m}]  \right), \tilde E_{k}^{[m+1/2]}\right)_{\partial T_k}
\end{eqnarray}
and
\begin{eqnarray} \label{fv23D}
\lefteqn{ \left(\mu \tilde{H}_{k}^{m+3/2},\tilde{H}_{k}^{m+3/2}\right)_{T_k}-\left(\mu \tilde{H}_{k}^{m+1/2},\tilde{H}_{k}^{m+1/2}\right)_{T_k}
= }\nonumber\\
&&-2\Delta t \left( \nabla \times \tilde E_{k}^{m+1} ,\tilde{H}_{k}^{[m+1]}\right)_{T_k}\nonumber\\
&&+ 2\Delta t \left( \frac{1}{Y^++Y^-} n \times \left(Y^+ [\tilde E^{m+1}]+\alpha n \times [\tilde H^{m+1/2}]\right),
\tilde H_{k}^{[m+1]}\right)_{\partial T_k}.  \nonumber \\\end{eqnarray}

Using the identity,
$$
\left( \nabla \times \tilde E_{k}^{m+1} ,\tilde{H}_{k}^{[m+1]}\right)_{T_k}= \left(\nabla \times \tilde{H}_{k}^{[m+1]},\tilde{E}_k^{m+1}\right)_{T_k}+ \left(    n \times \tilde{E}_{k}^{m+1} , \tilde{H}_{k}^{[m+1]} \right)_{\partial T_k},
$$
summing (\ref{fv13D}) and  (\ref{fv23D})  from $m=0$ to $m=M-1$, and integrating by parts, we get
\begin{eqnarray}
&&\left(\epsilon \tilde{E}_k^{M},\tilde{E}_k^{M}\right)_{T_k}+\left(\mu \tilde{H}_{k}^{M+1/2},\tilde{H}_{k}^{M+1/2}\right)_{T_k}=
\left(\epsilon \tilde{E}_k^{0},\tilde{E}_k^{0}\right)_{T_k}+\left(\mu \tilde{H}_{k}^{1/2},\tilde{H}_{z_k}^{1/2}\right)_{T_k}\nonumber\\
&&\qquad +\Delta t \left(\nabla \times \tilde{H}_{k}^{1/2},\tilde{E}_k^{0}\right)_{T_k}
- \Delta t \left(\nabla \times \tilde{H}_{k}^{M+1/2},\tilde{E}_k^{M}\right)_{T_k}
+2 \Delta t \sum_{m=0}^{M-1}   A_k^m,\label{cond3D}
 \end{eqnarray}
where
\begin{eqnarray*}
A_k^m &=&- \left( \frac{1}{Z^++Z^-} n \times \left(Z^+ [\tilde H^{m+1/2}]-\alpha n \times [\tilde E^{m}]  \right), \tilde E_{k}^{[m+1/2]}\right)_{\partial T_k}\nonumber\\
&&\left( \frac{1}{Y^++Y^-} n \times \left(Y^+ [\tilde E^{m+1}]+\alpha n \times [\tilde H^{m+1/2}]\right),
\tilde H_{k}^{[m+1]}\right)_{\partial T_k}\\
&&-\left(    n \times \tilde{E}_{k}^{m+1} , \tilde{H}_{k}^{[m+1]} \right)_{\partial T_k}.
 \end{eqnarray*}
 
 Let us consider the following decomposition
 $$\sum_{T_k\in \mathcal T_h}A_k^m = B_1^m + B_2^m,$$
 where 
    \begin{eqnarray}
B_1^m &  =&\sum_{f_{ik} \in F^{int} } \int_{f_{ik}} \Bigg(  \frac{-1}{Z_i+Z_k} n_{ki}\times \left(Z_i [\tilde H_{k}^{m+1/2}]-\alpha n_{ki}\times [\tilde E_{k}^{m}]  \right)\cdot
\tilde E_{k}^{[m+1/2]}\nonumber\\
&&\quad  + \frac{1}{Z_i+Z_k}n_{ki} \times \left(Z_k [\tilde H_{i}^{m+1/2}]+\alpha n_{ki}\times [\tilde E_{i}^{m}]\right)
\cdot \tilde E_{i}^{[m+1/2]}\nonumber\\
&&\quad  + \frac{1}{Y_i+Y_k} n_{ki}\times \left( Y_i [\tilde E_{k}^{m+1}]+\alpha n_{ki}\times [\tilde H_{k}^{m+1/2}]\right)\cdot
\tilde H_{k}^{[m+1]}\nonumber\\
&&\quad  - \frac{1}{Y_i+Y_k} n_{ki}\times \left( Y_k [\tilde E_{i}^{m+1}]-\alpha n_{ki}\times [\tilde H_{i}^{m+1/2}]\right)\cdot
\tilde H_{i}^{[m+1]}\nonumber\\
&&\quad  - n_{ki}\times \tilde{E}_{k}^{m+1} \cdot \tilde{H}_{k}^{[m+1]}
+ n_{ki}\times \tilde{E}_{i}^{m+1} \cdot \tilde{H}_{i}^{[m+1]} \Bigg)\, ds,\label{B1m3D}
 \end{eqnarray}
 and
 \begin{eqnarray}
B_2^m&=&\sum_{f_{k} \in F^{ext} } \int_{f_{k}} \Bigg(  - \frac{1}{2 Z_k} n_k \times \left(Z_k [\tilde H_k^{m+1/2}]-\alpha n_k \times [\tilde E_k^{m}]  \right) \cdot \tilde E_{k}^{[m+1/2]}\nonumber\\
&& + \frac{1}{2Y_k} n_k \times \left(Y_k[\tilde E_k^{m+1}]+\alpha n_k \times [\tilde H_k^{m+1/2}]\right) \cdot
\tilde H_{k}^{[m+1]}-   n_k \times \tilde{E}_{k}^{m+1} \cdot \tilde{H}_{k}^{[m+1]}\Bigg)  \, ds .\label{B2m3D}
 \end{eqnarray}

 We will now estimate $B_1^m$ and $B_2^m$. 
 In what follows, we use the inequalities:
 \begin{equation}\label{cross_dot}
  u\times v \cdot w = - u\times w\cdot v,
  \end{equation}
  and
     \begin{equation}\label{cross_cross}
  u\times( v \times w) = v  (u\cdot w)-w (u\cdot v).
  \end{equation}
  
 \begin{lemma}\label{B13D}
 Let $B_1^m$ be defined by (\ref{B1m3D}).%
 Then
     \begin{eqnarray}
 \sum_{m=0}^{M-1}B_1^m 
&  \le& \sum_{f_{ik} \in F^{int} } \int_{f_{ik}} \Bigg(\frac{\alpha}{4(Z_i+Z_k)} [\tilde E_{k}^{M}] \cdot    [\tilde E_{k}^{M}]+ \frac{\alpha}{4(Y_i+Y_k)}[\tilde H_{k}^{M+1/2}] \cdot    [\tilde H_{k}^{M+1/2}]\nonumber\\
&&+\frac{Z_k}{2(Z_i+Z_k)} \left(n_{ki}\times [\tilde H_k^{1/2}]\cdot [E_k^0] - n_{ki}\times [\tilde H_k^{M+1/2}]\cdot [\tilde E_k^{M}]\right)\nonumber\\
&&+\frac12\left(n_{ki}\times [\tilde H_k^{M+1/2}]\cdot \tilde E_k^{M}  - n_{ki}\times [\tilde H_k^{1/2}]\cdot \tilde E_k^{0}\right)   \Bigg)\, ds.\nonumber
\end{eqnarray}
\end{lemma}

\begin{proof}
Summing from $m=0$ to $m=M-1$, usin   (\ref{ZY1}),  (\ref{ZY2}), (\ref{cross_dot}) and (\ref{cross_cross}), we obtain
    \begin{eqnarray}
 \sum_{m=0}^{M-1}B_1^m   &=& \sum_{f_{ik} \in F^{int} } \int_{f_{ik}} \Bigg(\frac{\alpha}{4(Z_i+Z_k)}\Bigg( - [\tilde E_{k}^{0}]^T  \left(I-n_{ki} n_{ik}^T\right)  [\tilde E_{k}^{0}]  + [\tilde E_{k}^{M}]^T  \left(I-n_{ki} n_{ik}^T\right)   [\tilde E_{k}^{M}] \nonumber\\
&&  - \sum_{m=0}^{M-1} \left(  [\tilde E_{k}^{m}]+  [\tilde E_{k}^{m+1}]\right) ^T \left(I-n_{ki} n_{ik}^T\right)  \left(  [\tilde E_{k}^{m}]+  [\tilde E_{k}^{m+1}]\right)  \Bigg)\nonumber\\
&&+ \frac{\alpha}{4(Y_i+Y_k)}\Bigg( - [\tilde H_{k}^{1/2}]^T  \left(I-n_{ki} n_{ik}^T\right)  [\tilde H_{k}^{1/2}]  + [\tilde H_{k}^{M+1/2}]^T  \left(I-n_{ki} n_{ik}^T\right)   [\tilde H_{k}^{M+1/2}] \nonumber\\
&& - \sum_{m=0}^{M-1} \left(  [\tilde H_{k}^{m+1/2}]+  [\tilde H_{k}^{m+3/2}]\right) ^T \left(I-n_{ki} n_{ik}^T\right)  \left(  [\tilde H_{k}^{m+1/2}]+  [\tilde H_{k}^{m+3/2}]\right)  \Bigg)\nonumber\\
&& +\frac{Z_k}{2(Z_i+Z_k)} \left(n_{ki}\times [\tilde H_k^{1/2}]\cdot [E_k^0] - n_{ki}\times [\tilde H_k^{M+1/2}]\cdot [\tilde E_k^{M}]\right)\nonumber\\
&&+\frac12\left(n_{ki}\times [\tilde H_k^{M+1/2}]\cdot \tilde E_k^{M}  - n_{ki}\times [\tilde H_k^{1/2}]\cdot \tilde E_k^{0}\right)   \Bigg)\, ds\nonumber
\end{eqnarray}
where $I$ is an identity matrix. Since $I-n_{ki} n_{ik}^T$ is a positive semidefinite matrix,
    \begin{eqnarray}
\sum_{m=0}^{M-1}B_1^m   &\le& \sum_{f_{ik} \in F^{int} } \int_{f_{ik}} \Bigg(\frac{\alpha}{4(Z_i+Z_k)}\Bigg( [\tilde E_{k}^{M}]^T  \left(I-n_{ki} n_{ik}^T\right)   [\tilde E_{k}^{M}]\nonumber\\
&&\quad + \frac{\alpha}{4(Y_i+Y_k)}\Bigg([\tilde H_{k}^{M+1/2}]^T  \left(I-n_{ki} n_{ik}^T\right)   [\tilde H_{k}^{M+1/2}]\nonumber\\
&&\quad+\frac{Z_k}{2(Z_i+Z_k)} \left(n_{ki}\times [\tilde H_k^{1/2}]\cdot [E_k^0] - n_{ki}\times [\tilde H_k^{M+1/2}]\cdot [\tilde E_k^{M}]\right)\nonumber\\
&&\quad +\frac12\left(n_{ki}\times [\tilde H_k^{M+1/2}]\cdot \tilde E_k^{M}  - n_{ki}\times [\tilde H_k^{1/2}]\cdot \tilde E_k^{0}\right)   \Bigg)\, ds.\nonumber
\end{eqnarray}
The proof follows from the fact that, since the matriz $I-n_{ki} n_{ik}^T$ is an orthogonal projector, $x^T (I-n_{ki} n_{ik}^T)x\le x\cdot x$, for all vector $x$. \end{proof}

 \begin{lemma}\label{B23D}
 Let $B_2^m$ be defined by (\ref{B2m3D}).%
 Then
  \begin{eqnarray*}
 \sum_{m=0}^{M-1} B_2^m 
&\leq &\sum_{f_{k} \in F^{ext} } \int_{f_{k}}\Bigg(\frac{\beta_1}{4 Z_k}\Bigg( -\left(n_{k} \times  \tilde E_{k}^{0} \right) \cdot \left(n_{k} \times  \tilde  E_{k}^{0} \right)+ \left(n_{k} \times  \tilde  E_{k}^{M} \right) \cdot \left(n_{k} \times  \tilde  E_{k}^{M} \right)\Bigg)  \nonumber\\
&&+ \frac{\beta_3}{4 Y_k}\Bigg( -\left(n_{k} \times \tilde  H_{k}^{1/2} \right) \cdot \left(n_{k} \times \tilde   H_{k}^{1/2} \right)+ \left(n_{k} \times  \tilde  H_{k}^{M+1/2} \right) \nonumber\\
&& \cdot \left(n_{k} \times \tilde  H_{k}^{M+1/2} \right)\Bigg) 
+ \frac{\beta_2}{2} \left(n_k \times \tilde{E}_{k}^{0} \cdot \tilde{H}_{k}^{1/2}-  n_k \times \tilde{E}_{k}^{M} \cdot \tilde{H}_{k}^{M+1/2} \right)\Bigg)
\, ds,
 \end{eqnarray*}
 where $\beta_1=\alpha, \beta_2=0$ for PEC, $\beta_1=0, \beta_2=1$, $\beta_3=\alpha$ for PMC,  and $\beta_1=\beta_2=\beta_3=\frac12$    for Silver-M\"uller boundary conditions.
\end{lemma}

\begin{proof} First we consider PEC boundary
conditions. We have
   \begin{eqnarray*}
B_2^m
= \sum_{f_{k} \in F^{ext} } \int_{f_{k}}   \frac{1}{ Z_k} n_k \times \left(\alpha n_k \times \tilde E_k^{m}  \right) \cdot \tilde E_{k}^{[m+1/2]} \, ds,
 \end{eqnarray*}
  and then
      \begin{eqnarray*}
 \sum_{m=0}^{M-1} B_2^m
&\leq & \sum_{f_{k} \in F^{ext} } \int_{f_{k}}\frac{\alpha}{4 Z_k}\Bigg( -\left(n_{k} \times  \tilde E_{k}^{0} \right) \cdot \left(n_{k} \times  \tilde  E_{k}^{0} \right)+ \left(n_{k} \times  \tilde  E_{k}^{M} \right) \cdot \left(n_{k} \times  \tilde  E_{k}^{M} \right)\Bigg) \, ds.
 \end{eqnarray*}

For PMC  boundary
conditions, results
\begin{eqnarray*}
B_2^m&=&\sum_{f_{k} \in F^{ext} } \int_{f_{k}}  \Bigg(-  n_k \times \tilde H_k^{m+1/2} \cdot \tilde E_{k}^{[m+1/2]}+ \frac{\alpha}{Y_k} n_k \times \left(n_k \times \tilde H_k^{m+1/2} \right) \cdot
\tilde H_{k}^{[m+1]}\nonumber\\
&&\quad-   n_k \times \tilde{E}_{k}^{m+1} \cdot \tilde{H}_{k}^{[m+1]}\Bigg)  \, ds,
 \end{eqnarray*}
  and then
  \begin{eqnarray*}
 \sum_{m=0}^{M-1} B_2^m
&\leq &\sum_{f_{k} \in F^{ext} } \int_{f_{k}}  \Bigg(\frac{\alpha}{4 Y_k}\Bigg( -\left(n_{k} \times \tilde  H_{k}^{1/2} \right) \cdot \left(n_{k} \times \tilde   H_{k}^{1/2} \right)+ \left(n_{k} \times  \tilde  H_{k}^{M+1/2} \right) \nonumber\\
&& \cdot \left(n_{k} \times \tilde  H_{k}^{M+1/2} \right)\Bigg) 
+ \frac{1}{2} n_k \times \tilde{E}_{k}^{0} \cdot \tilde{H}_{k}^{1/2}-  \frac{1}{2} n_k \times \tilde{E}_{k}^{M} \cdot \tilde{H}_{k}^{M+1/2}\Bigg)
\, ds.
 \end{eqnarray*}
 
For Silver-M\"uler boundary conditions, we have
\begin{eqnarray*}
 B_2^m&=
&\sum_{f_{k} \in F^{ext} } \int_{f_{k}} \Bigg(  - \frac{1}{2 Z_k} n_k \times \left(Z_k \tilde H_k^{m+1/2}- n_k \times \tilde E_k^{m}\right) \cdot \tilde E_{k}^{[m+1/2]} \nonumber\\
&&+ \frac{1}{2Y_k} n_k \times \left(Y_k \tilde E_k^{m+1}+ n_k \times \tilde H_k^{m+1/2} \right) \cdot
\tilde H_{k}^{[m+1]}-   n_k \times \tilde{E}_{k}^{m+1} \cdot \tilde{H}_{k}^{[m+1]}\Bigg)  \, ds
\nonumber\\
&=&\sum_{f_{k} \in F^{ext} } \int_{f_{k}}\Bigg(  - \frac{1}{2 Z_k} n_k \times \left(Z_k \tilde H_k^{m+1/2}- n_k \times \tilde E_k^{m}\right) \cdot \tilde E_{k}^{[m+1/2]} \nonumber\\
&&+ \frac{1}{2Y_k} n_k \times \left( n_k \times \tilde H_k^{m+1/2} \right) \cdot
\tilde H_{k}^{[m+1]}.-  \frac{1}{2} n_k \times \tilde{E}_{k}^{m+1} \cdot \tilde{H}_{k}^{[m+1]}\Bigg) \, ds,
 \end{eqnarray*}
 and then
  \begin{eqnarray*}
 \sum_{m=0}^{M-1} B_2^m
&\leq& \sum_{f_{k} \in F^{ext} } \int_{f_{k}}\Bigg(\frac{1}{8 Z_k}\Bigg( -\left(n_{k} \times  \tilde E_{k}^{0} \right) \cdot \left(n_{k} \times  \tilde  E_{k}^{0} \right)+ \left(n_{k} \times  \tilde  E_{k}^{M} \right) \cdot \left(n_{k} \times  \tilde  E_{k}^{M} \right)\Bigg)  \nonumber\\
&&+ \frac{1}{8 Y_k}\Bigg( -\left(n_{k} \times \tilde  H_{k}^{1/2} \right) \cdot \left(n_{k} \times \tilde   H_{k}^{1/2} \right)+ \left(n_{k} \times  \tilde  H_{k}^{M+1/2} \right) \nonumber\\
&& \cdot \left(n_{k} \times \tilde  H_{k}^{M+1/2} \right)\Bigg) 
+ \frac{1}{4} n_k \times \tilde{E}_{k}^{0} \cdot \tilde{H}_{k}^{1/2}-  \frac{1}{4} n_k \times \tilde{E}_{k}^{M} \cdot \tilde{H}_{k}^{M+1/2}\Bigg)
\, ds,
 \end{eqnarray*}
 which concludes the proof.
\end{proof}

\begin{theorem} \label{stab_theorem3D} Let us consider the leap-frog DG method (\ref{Fdis1})--(\ref{Fdis3}) complemented with the discrete boundary conditions defined in Section \ref{DG_method}. If the time step $\Delta t$ is such that
\begin{equation}\label{stab_cond3D}
\Delta t < \frac{\min\{\munderbar \epsilon, \munderbar\mu\}}{\max\{{C}_E, C_H\}} \min\{h_k\},
\end{equation}
where
$${C}_E=\frac12 C_{inv} N^2+C_{\tau}^2 (N+1)(N+3)\left(3+\frac{\beta_2}{2}+\frac{\alpha+\beta_1}{2\min\{Z_k\}}\right),$$
$${C}_H= \frac12 C_{inv} N^2+C_{\tau}^2 (N+1)(N+3)\left(3+\frac{\beta_2}{2}+\frac{\alpha+\beta_3}{2\min\{ Y_k\}}\right), $$
 with $C_{\tau}$ defined by (\ref{inq_trace_3D}) of Lemma \ref{trace}  and $C_{inv}$ defined by (\ref{inq_inv}) of Lemma \ref{inverse_ineq}, and  $\beta_1=\alpha, \beta_2=0$ for PEC, $\beta_1=0, \beta_2=1$, $\beta_3=\alpha$ for PMC,  and $\beta_1=\beta_2=\frac12$, $\beta_3=1$ for Silver-M\"uller boundary conditions, then the method is stable.
\end{theorem}
\begin{proof}

As for the 2D case, from (\ref{cond3D}) and the previous lemmata, considering  the Cauchy-Schwarz's and triangular inequality inequality, taking into account that $Z_i/(Z_i+Z_k)<1$, 
%
and using the inequality (\ref{inq_trace_3D}) of Lemma \ref{trace} and the inequality  (\ref{inq_inv}) of Lemma \ref{inverse_ineq} (both in Appendix),  we get
\begin{eqnarray*}
\lefteqn{ \min\{\munderbar\epsilon,\munderbar\mu\}\left(\|\tilde{E}^{M}\|^2_{\Omega}+\|\tilde{H}_{z}^{M+1/2}\|^2_{\Omega}\right)\le\max\{\bar\epsilon, \bar\mu\}\left(\|\tilde{E}^{0}\|^2_{\Omega}+\|\tilde{H}_{z}^{1/2}\|^2_{\Omega}\right)}\\
&&+\frac{\Delta t}{2} C_{inv} N^2 \max\left\{h_k^{-1} \right\} \left(\|\tilde{H}_{z}^{1/2}\|_{\Omega}^2 + \|\tilde{E}^{0}\|_{\Omega}^2 +\| \tilde{H}_{z}^{M+1/2}\|_{\Omega}^2+ \|\tilde{E}^{M}\|_{\Omega}^2 \right)\\
&&+C_{\tau}^2 (N+1)(N+3)\Delta t\max\left\{h_k^{-1} \right\}\left(3+\frac{\beta_2}{2}+\frac{\alpha+\beta_1}{2\min\{Z_k\}}\right)\left( \| \tilde E^0\|_{\Omega}^2 +  \| \tilde E^M\|_{\Omega}^2\right)\\
&&+C_{\tau}^2 (N+1)(N+3)\Delta t\max\left\{h_k^{-1} \right\}\left(3+\frac{\beta_2}{2}+\frac{\alpha+\beta_3}{2\min\{ Y_k\}}\right) \left( \| \tilde H_{z}^{1/2}\|_{\Omega}^2 +  \| \tilde H_{z}^{M+1/2}\|_{\Omega}^2\right).
 \end{eqnarray*}
and so
 \begin{eqnarray*}
&&\left(\min\{\munderbar \epsilon, \munderbar\mu\}-\Delta t \max\left\{h_k^{-1} \right\}\max\{{C}_E,  C_H\}\right)\left(\|\tilde{E}^{M}\|^2_{L^2(\Omega)}+\|\tilde{H}_{z}^{M+1/2}\|^2_{L^2(\Omega)}\right)\\
&&\qquad \le \left(\max\{\bar \epsilon, \bar\mu\}+\Delta t \max\left\{h_k^{-1} \right\}\max\{{C}_E,  C_H\}\right)\left(\|\tilde{E}^{0}\|^2_{L^2(\Omega)}+\|\tilde{H}_{z}^{1/2}\|^2_{L^2(\Omega)}\right),
\end{eqnarray*}
which proves the result. 
\end{proof}


\section*{Acknowledgments}
This work was partially supported by the Centre for Mathematics of the University of Coimbra -- UID/MAT/00324/2013, funded by the Portuguese Government through FCT/MEC and co-funded by the European Regional Development Fund through the Partnership Agreement PT2020; and by the Portuguese Government through FCT/MEC under the project PTDC/SAU-ENB/119132/2010 and the BD grant SFRH/BD/51860/2012.


\appendix\section{Technical lemmata}

The  lemmata included  this section are technical tools needed to derive the stability conditions.

We consider the following trace inequalities (see {\it e.g.} \cite{riviere2008}).
\begin{lemma}\label{trace}
Let  $T_k$ be an element  of $\mathcal T_h$ with diameter $h_k$ and let $f_k$ be  an edge  or a face  of $T_k$.
There exists a positive constant $C$ independent of $h_k$ such that, for any  $u \in H^1({T_k})$,
\begin{equation} \label{inq_trace1}
\|u\|_{L^2(f_k)}\leq   C  \sqrt{ \frac{|f_k| }{|T_k|}} \left(\|u\|_{L^2(T_k)}+ h_k \|\nabla u\|_{L^2(T_k)}\right).
\end{equation}
Moreover, if $u$ is a polynomials of degree less than or equal to $N$, there exists a positive constant $C_{trace}$  independent of $h_k$ and $u$ but dependent on the polynomials degree $N$, such that
$$
 \|u\|_{L^2(f_k)} \leq C_{trace} \sqrt{\frac{|f_k|}{|T_k|}}  \|u\|_{L^2(T_k)}.
$$
An exact expression for the constant  $C_{trace}$ can be given as a function of the polynomials degree, and the following inequality holds
 for any  $u \in P_N(T_k)$
\begin{equation}\label{inq_trace_exp}
\mbox{in } 2D:  \quad  \|u\|_{L^2(f_k)} \leq \sqrt{\frac{(N+1)(N+2)}{2} \frac{|f_k|}{|T_k|}}  \|u\|_{L^2(T_k)},
 \end{equation}
 \begin{equation}\label{inq_trace_exp_3D}
\mbox{in } 3D:  \quad  \|u\|_{L^2(f_k)} \leq \sqrt{\frac{(N+1)(N+3)}{3} \frac{|f_k|}{|T_k|}}  \|u\|_{L^2(T_k)}.
 \end{equation}
 Consequently, there exists a positive constant $C_{\tau}$ independent of $h_k$ and $N$ but depent on the shape-regularity $h_k/\tau_k$, where $\tau_k$ is the diameter of  the largest inscribed ball  contained in $T_k$ (see (\ref{tau})), such that,  for any  $u \in P_N(T_k)$,
 \begin{equation}\label{inq_trace}
\mbox{in } 2D: \quad \|u\|_{L^2(\partial T_k)} \leq C_{\tau} \sqrt{(N+1)(N+2)}  h_k^{-1/2}  \|u\|_{L^2(T_k)},
 \end{equation}
 \begin{equation}\label{inq_trace_3D}
\mbox{in } 3D:  \quad  \|u\|_{L^2(\partial T_k)} \leq C_{\tau} \sqrt{(N+1)(N+3)}  h_k^{-1/2}  \|u\|_{L^2(T_k)}.
 \end{equation}
\end{lemma}

The next result is an inverse-type estimate (\cite{ciarlet1980, georgoulis2008}), where we present explicitly the dependence of the constant on the polynomials degree.
\begin{lemma}\label{inverse_ineq}
Let us consider $T_k  \in \mathcal T_h$  with diameter $h_k$. There exists a positive constant $C_{inv}$ independent of $h_k$ and  $N$ such that, for any $u \in P_N(T_k)$,
\begin{equation}\label{inq_inv}
 \|u\|_{H^q(T_k)}\leq C_{inv} N^{2q}  h_k^{-q}   \| u\|_{L^2(T_k)},
\end{equation}
where $q\geq 0$.
 \end{lemma}


\bibliographystyle {plain}
\bibliography{Max}

\begin{thebibliography}{10}

\bibitem{alvarez2014}
Jes\'{u}s \'{A}lvarez Gonz\'{a}lez.
\newblock {\em A Discontinuous {G}alerkin Finite Element Method for the
  Time-Domain Solution of {M}axwell Equations}.
\newblock PhD thesis, Universidad de Granada, 2014.

\bibitem{araujo2013}
A.~Ara\'ujo, S.~Barbeiro, L.~Pinto, F.~Caramelo, A.~L. Correia, M.~Morgado,
  P.~Serranho, A.~S.~C. Silva, and R.~Bernardes.
\newblock {M}axwell's equations to model electromagnetic wave's propagation
  through eye's structures.
\newblock {\em Proceedings of the 13th International Conference on
  Computational and Mathematical Methods in Science and Engineering,CMMSE 2013,
  Almeria, Ian Hamilton and Jes\'us Vigo-Aguiar Eds}, 1:121--129, June 2013.

\bibitem{bornwolf1999}
M~Born and E~Wolf.
\newblock {\em Principles of Optics}.
\newblock Cambridge University Press, 7 edition, 1999.

\bibitem{ciarlet1980}
P.~G. Ciarlet.
\newblock {\em The Finite Element Method for Elliptic Problems}.
\newblock Studies in mathematics and its applications. North-Holland,
  Amsterdam, New-York, 1980.

\bibitem{Dobson1999}
David~C Dobson.
\newblock An efficient method for band structure calculations in 2{D} photonic
  crystals.
\newblock {\em Journal of Computational Physics}, 149(2):363 -- 376, 1999.

\bibitem{fezoui2005}
L.~Fezoui, S.~Lanteri, S.~Lohrengel, and S.~Piperno.
\newblock Convergence and stability of a discontinuous {G}alerkin time-domain
  method for the 3{D} heterogeneous {M}axwell's equations on unstructured
  meshes.
\newblock {\em ESAIM: Mathematical Modeling and Numerical Analysis},
  39(6):1149--1176, 2005.

\bibitem{georgoulis2008}
E.~H. Georgoulis.
\newblock Inverse-type estimates on $hp$-finite element spaces and
  applications.
\newblock {\em Mathematics of Computation}, 77(261):201--219, 2008.

\bibitem{hesthaven2004}
J.~S. Hesthaven and T.~Warburton.
\newblock Discontinuous {G}alerkin methods for the time-domain {M}axwell's
  equations: An introduction.
\newblock {\em ACES Newsletter}, 19:10--29, 2004.

\bibitem{hesthaven2008}
J.S. Hesthaven and T.~Warburton.
\newblock {\em Nodal Discontinuous {G}alerkin Methods: Algorithms, Analysis,
  and Applications}.
\newblock Springer Publishing Company, Incorporated, 1st edition, 2008.

\bibitem{konig2010}
M.~K{\"o}nig, K.~Busch, and J.~Niegemann.
\newblock The discontinuous {G}alerkin time-domain method for {M}axwell's
  equations with anisotropic materials.
\newblock {\em Photonics and Nanostructures-Fundamentals and Applications},
  8(4):303--309, 2010.

\bibitem{Leonhardt110}
Ulf Leonhardt and Tom{\'a}{\v s} Tyc.
\newblock Broadband invisibility by non-euclidean cloaking.
\newblock {\em Science}, 323(5910):110--112, 2009.

\bibitem{li2012}
J.~Li, J.~W. Waters, and E.~A. Machorro.
\newblock An implicit leap-frog discontinuous {G}alerkin method for the
  time-domain {M}axwell's equations in metamaterials.
\newblock {\em Computer Methods in Applied Mechanics and Engineering},
  223:43--54, 2012.

\bibitem{lu2004}
T.~Lu, P.~Zhang, and W.~Cai.
\newblock Discontinuous {G}alerkin methods for dispersive and lossy {M}axwell's
  equations and {PML} boundary conditions.
\newblock {\em Journal of Computational Physics}, 200(2):549--580, 2004.

\bibitem{Mikhailov2007}
S.~A. Mikhailov and K.~Ziegler.
\newblock New electromagnetic mode in graphene.
\newblock {\em Phys. Rev. Lett.}, 99:016803, Jul 2007.

\bibitem{riviere2008}
B.~Rivi\'{e}re.
\newblock {\em Discontinuous Galerkin Methods for Solving Elliptic and
  Parabolic Equations: Theory and Implementation}.
\newblock SIAM, Philadelphia, PA, USA, 2008.

\bibitem{Santos2015}
M.~Santos, A.~Ara{\'u}jo, S.~Barbeiro, F.~Caramelo, A.~Correia, M.~I. Marques,
  L.~Pinto, P.~Serranho, R.~Bernardes, and M.~Morgado.
\newblock Simulation of cellular changes on optical coherence tomography of
  human retina.
\newblock In {\em 2015 37th Annual International Conference of the IEEE
  Engineering in Medicine and Biology Society (EMBC)}, pages 8147--8150, Aug
  2015.

\bibitem{yee1966}
K.~S. Yee.
\newblock Numerical solution of initial boundary value problems involving
  {M}axwell's equations in isotropic media.
\newblock {\em IEEE Transaction on Antennas and Propagation}, 14(3):302--307,
  1966.

\bibitem{Yeh2009}
Pochi Yeh and Claire Gu.
\newblock {\em Optics of Liquid Crystal Displays}.
\newblock Wiley Publishing, 2nd edition, 2009.

\end{thebibliography}

\end{document}